\documentclass[11pt]{article} 
\usepackage[applemac]{inputenc}
\usepackage[T1]{fontenc}
\usepackage{lmodern}
\usepackage[english]{babel}

\usepackage{amsmath,amsfonts,amssymb,amsthm,mathrsfs,amsrefs}

\renewcommand{\leq}{\leqslant}
\renewcommand{\geq}{\geqslant}
\usepackage[cmintegrals,libertine]{newtxmath}
\usepackage[cal=euler]{mathalfa}
\usepackage[mono=false]{libertine}
\useosf
\usepackage[a4paper,vmargin={3.5cm,3.5cm},hmargin={2.5cm,2.5cm}]{geometry}
\usepackage[font={small,sf}, labelfont={sf,bf}, margin=1cm]{caption}
\usepackage[pdftex,colorlinks=true]{hyperref}
\usepackage[expansion]{microtype}
\usepackage[pdftex]{color,graphicx}

\usepackage{stackrel}
\usepackage{soul} 

\theoremstyle{plain}
\newtheorem{theorem}{Theorem}
\newtheorem{corollary}[theorem]{Corollary}
\newtheorem{proposition}[theorem]{Proposition}
\newtheorem{lemma}[theorem]{Lemma}
\theoremstyle{definition}
\newtheorem{definition}[theorem]{Definition}

\newtheorem{remark}[theorem]{Remark}

\newtheorem*{thm}{Theorem}

\hypersetup{
    pdftitle    = {UIPT}}

\begin{document}

\title{\bf The skeleton of the UIPT, seen from infinity}
\author{\textsc{Nicolas Curien}\footnote{Universit\'e Paris-Saclay and Institut Universitaire de France. E-mail: \href{mailto:nicolas.curien@gmail.com}{nicolas.curien@gmail.com}.} \, and \textsc{Laurent M\'enard}\footnote{Universit\'e Paris Nanterre. Email: \href{mailto:laurent.menard@normalesup.org}{laurent.menard@normalesup.org}}}
\date{\today}
\maketitle


\begin{abstract}
We prove that geodesic rays in the Uniform Infinite Planar Triangulation (UIPT) coalesce in a strong sense using the skeleton decomposition of random triangulations discovered by Krikun. This implies the existence of a unique horofunction measuring distances from infinity in the UIPT. We then use this horofunction to define the skeleton ``seen from infinity'' of the UIPT and relate it to a simple Galton--Watson tree conditioned to survive, giving a new and particularly simple construction of the UIPT. Scaling limits of perimeters and volumes of horohulls within this new decomposition are also derived, as well as a new proof of the $2$-point function formula for random triangulations in the scaling limit due to Ambj\o rn and Watabiki.
\end{abstract}

\bigskip




\begin{figure}[!h]
 \begin{center}
 \includegraphics[width=0.9\linewidth]{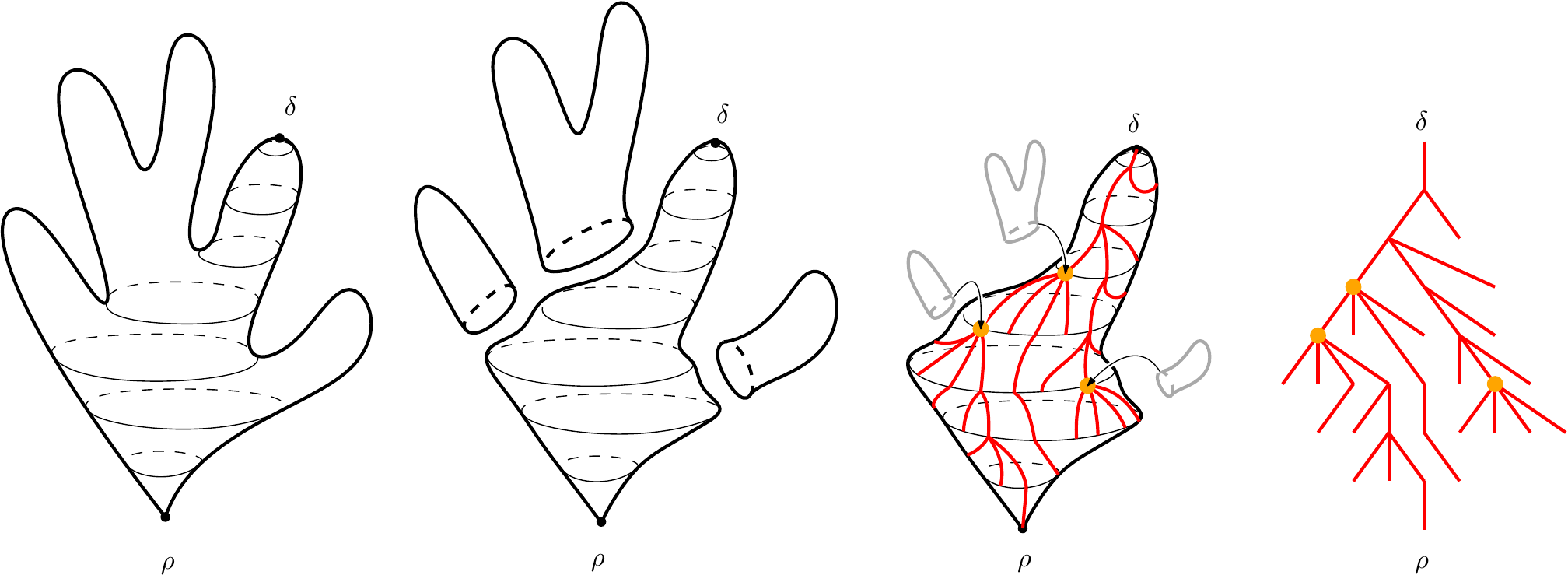}
 \caption{ \label{fig:nutshell}A rough view of the skeleton decomposition of a triangulation with origin $\rho$ and target $\delta$: From left to right we see a triangulation with two distinguished points and the cycles at heights separating $\rho$ from $\delta$. The triangulation is represented as a ``cactus'' where heights measure distances from the origin vertex $\rho$. On the second figure we have cut out the parts of the triangulation which do not belong to these cycles. These parts are separately encoded and attached to vertices of the red tree which describes the genealogy of the cycles. The last figure is the tree alone.}
 \end{center}
 \end{figure}

\section{Introduction} 
Since its introduction by Angel \& Schramm \cite{AS}, the Uniform Infinite Planar Triangulation (UIPT) is a central character in the theory of random planar maps and despite intensive research, some of its basics properties remain mysterious. The skeleton decomposition of random triangulations was introduced by Krikun in \cite{Kr} for the UIPT and later extended to the quadrangular case \cites{Kr,LGL17}. It has recently been proved very useful in order to study the geometry of random triangulations  \cites{CLGfpp,AR,LGL17,Budzinski18,M}. The present paper is another illustration that the skeleton representation can serve as a convenient substitute of the classical Schaeffer-type constructions to study the geometry of random triangulations.

\paragraph{Skeleton decomposition in a nutshell.} Roughly speaking the skeleton decomposition can be described as follows: Consider a planar triangulation with two distinguished vertices, an origin $\rho$ and a target $\delta$. In the case of an infinite one-ended triangulation, the target $\delta$ can be considered as located at infinity. We can then ``slice'' the triangulation according to distances to the origin, called heights, and consider the cycles of vertices surrounding $\delta$. The skeleton decomposition is a way to associate to these nested cycles a tree structure ``coming down from $\delta$'' encoding these cycles. The parts of the triangulation which are cut out during this operation (so-called baby universes in physics) form triangulations with a simple boundary  which are attached to the vertices of the tree. Obviously this description is heuristic  and the mathematically precise (but longer) construction is presented in Section \ref{sec:skel}.

Krikun developed this decomposition in the case of triangulations without self-loops  in \cite{Kr} but it turns out that the decomposition is even simpler in the case of general triangulations as shown in \cite{CLGfpp}. In the case of the UIPT the distribution of the resulting tree $ \mathsf{Krik}$ is simple and is related to that of a Galton--Watson tree whose  offspring distribution $\theta$ has generating function given by
 \begin{eqnarray} \varphi(z)= \sum_{i=0}^\infty z^i \theta(i) = 1 - \left( 1+ \frac{1}{ \sqrt{1-z}}\right)^{-2}, \qquad z \in [0,1] \label{eq:offspring},  \end{eqnarray}
see \cites{CLGfpp,M}. In particular, this offspring distribution is critical and in the domain of attraction of the totally asymmetric $3/2$-stable distribution. The triangulations attached to the vertices of the tree $ \mathsf{Krik}$ to complete the description of the UIPT are critical Boltzmann triangulations. These Boltzmann triangulations can be viewed as simple building blocks for random triangulations and are often met in several flavors. They usually have a global constraint -- \emph{e.g.} here being a triangulation with a simple boundary -- and the probability of a given triangulation $\mathfrak t$ satisfying the constraint is proportional to $x_c^{|\mathfrak t|}$, where $\mathfrak t$ denotes the number of vertices of $\mathfrak t$ and $x_c$ is the radius of convergence of the generating series of triangulations counted by vertices (see Section \ref{sec:combi}).

The similarities between this skeleton decomposition approach \emph{\`a la} Krikun and the slice decomposition developed by Guitter \cite{G17} (see also \cite[Appendix A]{BG12}) are worth mentioning. Indeed, the slices of Guitter can roughly be interpreted at the triangulations coded by a single tree in the skeleton decomposition and in both cases the fact that we can obtain exact discrete formulas stem from the ``miracle'' that a certain iterate can be explicitly computed see our \eqref{eq:iterPhit} and \cite[Eq (3)]{G17}. Establishing an exact correspondence between the two approaches could be fruitful.

\medskip

The goal of the present work is to use the skeleton decomposition of the UIPT to extend the results of \cite{CMM10} obtained in the case of the Uniform Infinite Planar Quadrangulation (UIPQ) concerning distances from infinity and coalescence of geodesic rays. En route we develop another skeleton representation, where the distances are now measured from the target vertex $\delta$ which become distances from infinity in the case of the UIPT. This new combinatorial tool enables to perform interesting calculations such as a quick derivation of the two-point function of Ambj\o rn and Watabiki \cite{ADJ}  as well as computing  scaling limits of the horohull process which was only partially studied in \cite{CMM10}.

\paragraph{Results.} The phenomenon of \emph{geodesic coalescence} is by now standard in the theory of random planar maps. This was first proved by Le Gall \cite{LG09} in the case of the Brownian sphere (see also \cite{Bet16} for the case of other surfaces) and also observed at a discrete level in the case of the UIPQ in \cite{CMM10}. In this work, we will focus on the UIPT (type I: loops and multiple edges are allowed) which we denote by $ \mathbf{T}_{\infty}$ with origin vertex $\rho$. We prove a similar phenomenon, namely the fact that all infinite discrete geodesics starting within a bounded distance from the origin of the map must pass through an infinite sequence of ``cut-vertices'', see Section \ref{sec:coalescence} and in particular Theorem \ref{thm:scale}. This enables us to extend the main result of \cite{CMM10} to the UIPT:

\begin{theorem}[Distances from infinity] \label{thm:CMM} Almost surely, for any vertex $u \in \mathsf{V}(\mathbf \mathbf{T}_{\infty})$ the following limit exists
$$ \ell(u) = \lim_{z \to \infty} \left( \mathrm{d_{gr}}(u, z) - \mathrm{d_{gr}}(\rho,z)\right),$$
where
 $z \to \infty$ means that $z$ eventually leaves any finite set of vertices in the UIPT.
\end{theorem}
In words, the function $\ell$, called the horofunction, measures the relative distances to $\infty$ of vertices compared to the origin (horodistances). This shows that the geodesic boundary -- or Gromov boundary -- of the UIPT consists of a single point as in the case of the UIPQ (see \cite[Section 2.5.1]{CMM10}). As mentioned above, the geodesic coalescence properties and the above result  was proved  \cite{CMM10}  in the case of the UIPQ using two different Schaeffer-type constructions. In this work we rather use the skeleton decomposition to give a straightforward and more geometric proof of these facts (Section \ref{sec:coalescence}). 
\bigskip

We then use the horofunction $\ell$ to construct another skeleton decomposition of the UIPT by measuring distances from infinity rather than from the origin as in \cite{Kr}. More precisely,  for $r \geq 1$ we consider the horohull $ \overline{H}_{r}( \mathbf{T}_{\infty})$ as the set of all faces which we can reach from the origin $\rho$ by a path that stays at horodistance more than $-r+1$ from infinity. The boundaries $\partial \overline{H}_{r}( \mathbf{T}_{\infty})$ then slice the UIPT into layers and we can perform a similar skeleton decomposition, see Figure \ref{fig:skelinfinity} and Section \ref{sec:skelinfinity} for details. This gives rise to a tree $\mathsf{Skel}$ now climbing upwards from the origin in the UIPT (compared to $ \mathsf{Krik}$ which was growing downards from $\infty$). As before the vertices of the tree carry random triangulations with a boundary. Recall from \eqref{eq:offspring} the definition of $\theta$ and let $\nu$ be the offspring distribution whose generating function is given by 
 \begin{eqnarray} \label{def:nu}
 \psi(z) := \sum_{i=0}^{\infty} \nu(i) z^{i} = 1- 2 \frac{1-z}{(1+ 2 \sqrt{1-z})(1+ \sqrt{1-z})}, \qquad z \in [0,1].  \end{eqnarray}

\begin{figure}[!h]
 \begin{center}
 \includegraphics[width=12cm]{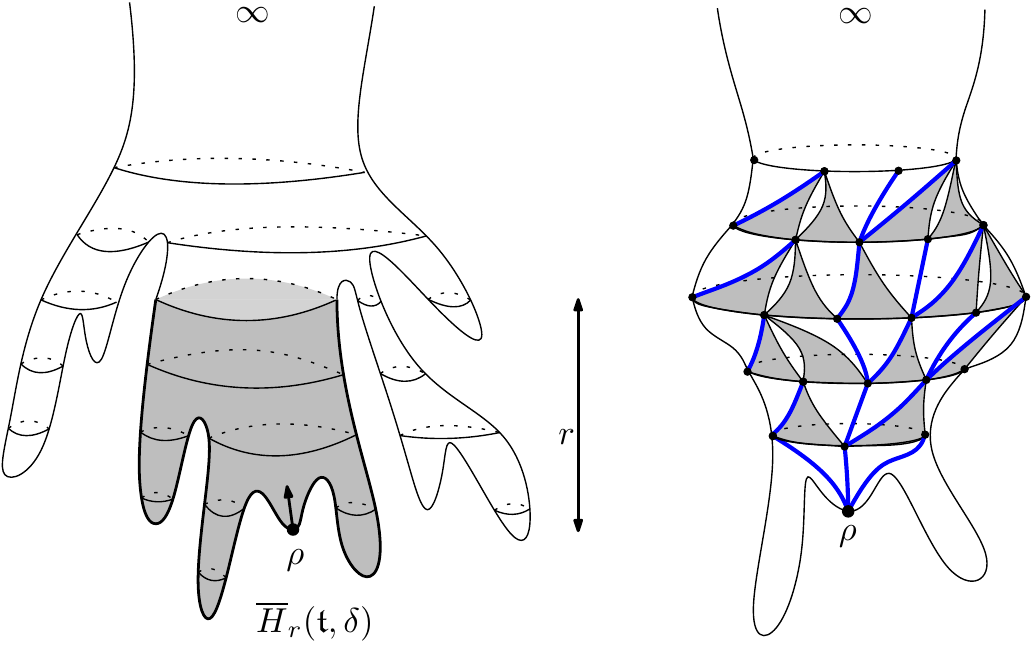}
 \caption{ \label{fig:skelinfinity} Left: The layering of the UIPT is done using the function $\ell$ measuring distances from infinity. Right: The skeleton decomposition of the slices separating the origin and infinity. Notice that contrary to the initial skeleton decomposition of Krikun \cite{Kr}, the blue tree $\mathsf{Skel}$ coding the skeleton now grows upwards in $ \mathbf{T}_{\infty}$.}
 \end{center}
 \end{figure}
We state here a rather unprecise version of our forthcoming Theorem \ref{thm:skeletoninfinity}
\begin{thm}[Skeleton decomposition from infinity] The infinite tree $ \mathsf{Skel}$ has the law of a (modified) Galton--Watson tree with offspring distribution $\theta$ and where the offspring distribution of the ancestor is $\nu$, conditioned to survive. Conditionally  on $ \mathsf{Skel}$ the triangulations attached to the vertices of the tree are independent and
\begin{itemize}
\item  vertices $v$ different from the ancestor carry critical Boltzmann triangulations of the $(c(v)+2)$-gon where $c(v)$ is the number of children of $v$ in $ \mathsf{Skel}$;
\item the origin carries a special random triangulation with a distinguished face of degree $c({\rho})$ called the critical Boltzmann $\delta-$cap of perimeter $c(\rho)$ whose law is described in Section \ref{sec:caps}.
\end{itemize}
\end{thm}

To prove the above theorem we first work in the setting of finite triangulations. As said above, Krikun's original skeleton decomposition is performed by measuring distances from the origin $\rho$ and then encoding the (descending) genealogy of cycles of vertices at the same height separating $\delta$ and $\rho$. We call this procedure the $\rho$-skeleton decomposition. Clearly, one can exchange the roles of the two vertices $\delta$ and $\rho$ and consider the $\delta$-skeleton decomposition. Actually the roles of $\rho$ and $\delta$ are not exactly symmetric in our context since we will work with rooted and pointed triangulations $( \mathfrak{t}, \vec{e}, \delta)$ where $\vec{e}$ is a distinguished oriented edge whose origin vertex is $\rho$ and $\delta$ is another distinguished vertex. The combinatorial description of the $\delta$-skeleton decomposition turns out to be particularly simple and enables us to perform a bunch of new computations using the results of \cite{M}. 
For example, we are able to compute in Section \ref{sec:2point} the joint distribution of the size of the map and distance between the origin and $\delta$  in a random Boltzmann pointed triangulation (that is a pointed triangulation whose weight is proportional to $x_c^{|T|}$ as mentioned above). More precisely, for any $h \geq 1$ we write $T_{h}^{\bullet}$ for a random critical Boltzmann pointed triangulation conditioned on the event where the distinguished point $\delta$ is at distance exactly $h$ from the origin of the map.
\begin{proposition}[Two point function of 2D gravity]
\label{prop:2point}
If $|T^{\bullet}_{h}|$ denotes the volume (i.e.~number of vertices) of $T^{\bullet}_{h}$ then we have
$$ \lim_{h\to \infty} \mathbb{E}\left[ e^{-\lambda h^{-4} |T^{\bullet}_{h}|} \right] =(6\lambda) ^{3/4} \frac{\cosh((6\lambda)^{1/4})}{\sinh^{3}( (6\lambda) ^{1/4})}.$$
\end{proposition}
The function $G(\Lambda) = \Lambda ^{3/4} \frac{\cosh(\Lambda^{1/4})}{\sinh^{3}( \Lambda ^{1/4})}$ is called the \emph{two-point function} of 2D quantum gravity in the physics literature, see Eq (4.357) in Section 4.7 of \cite{ADJ}. It has first been derived by Ambj\o rn and Watabiki using an early form of the peeling process \cite{AW}. See also the works of Bouttier, Di Francesco and Guitter \cites{BDFG03,BG12,DF05} based on bijective methods. It describes how distances and volumes are related to each other in the continuum limit. In a probabilistic framework it corresponds to the Laplace transform of the law of the volume of a ``Brownian Sphere with two marked points at distance $1$'' when sampled according the the excursion measure. This calculation has been done in  \cite[Proposition 4.6]{CLG16} using Brownian snake techniques.

We  also compute the asymptotic number of triangulations pointed at a given distance of the origin:
\begin{proposition} \label{prop:UIPTdisth}
The number of triangulations with $n$ vertices pointed at distance $h \geq 1$ of the origin vertex is asymptotic as $n \to \infty$ to 
\[ \frac{3}{4 \sqrt \pi} \frac{\sqrt{6}}{2835} \frac{-1+16h+80h^2+152h^3+138h^4+60h^5+10h^6}{h(h+1)(h+2)}
\left( 12\sqrt{3} \right)^{n} n^{-5/2}.
\]
\end{proposition}
See \cite{BDFG03} and \cite[Lemma 2.5]{CD06} for a similar result in the case of quadrangulations using bijective methods. We also give explicit expressions for the generating function of the perimeter and volume of the horohulls in the UIPT (see Section \ref{sec:skelinfinity} for the precise formulas). Taking the  scaling limit we in particular get the following

\begin{proposition}[Scaling limits for horohulls] We have 
\begin{align*}
\lim_{r \to \infty} 
\mathbb E &\Big[ 
\exp \left(- \lambda_1 r^{-4}|\overline{H}_{r} (\mathbf \mathbf{T}_{\infty})| - \lambda_2 r^{-2} |\partial \overline{H}_{r} (\mathbf \mathbf{T}_{\infty})| \right)\Big]\\
& \quad =
\frac
{
\left(\frac{2}{3} + \frac{\lambda_2}{(6 \lambda_1)^{1/2}} \right)^{-1/2}
\sinh \left( (6 \lambda_1)^{1/4}\right) + \cosh \left( (6 \lambda_1)^{1/4}\right)
}
{\Bigg( \left(\frac{2}{3} + \frac{\lambda_2}{(6 \lambda_1)^{1/2}} \right)^{1/2}
\sinh \left( (6 \lambda_1)^{1/4}\right) + \cosh \left( (6 \lambda_1)^{1/4}\right)\Bigg)^3}.
\end{align*}
\end{proposition}
The last proposition should be compared to the analogous result for perimeter and volume of hulls of standard metric balls in the UIPT \cite{M}, see also \cite{CLG16} for the continuous limit. In particular, the perimeter of the horohulls of radius $r$ properly rescaled by $r^{2}$ converges towards the variable $X$ with Laplace transform $ \mathbb{E}[e^{-\lambda X}] = (1+\sqrt{\lambda})^{-3}$, a fact already observed (via a different technique) in the quadrangulation case in \cite{CMM10}. In fact our analysis enables to understand the full scaling limit of the hull process $( r^{-2}  |\partial \overline{H}_{[rt]}|, r^{-4} | \overline{H}_{[rt]}|)_{t \geq 0}$, see Remark \ref{rem:fullscaling}. \medskip

\paragraph{Organization of the paper.} The paper is organized around its two main contributions. 

In Section \ref{sec:skeletongeneral} we recall the classical $\rho$-skeleton decomposition of triangulations due to Krikun and develop our new $\delta$-skeleton decomposition by measuring distances from the pointed vertex. The associated notion of horohulls is introduced at this occasion. When combined  in Section \ref{sec:gs} with the computations performed in \cite{M}, this new representation of pointed triangulations gives access to new distributional identities.

In Section \ref{sec:coalescence} we prove the geodesic coalescence property of the UIPT. Our main tool there is the good old $\rho$-skeleton decomposition of Krikun of the UIPT \cite{Kr}. The idea is to design a specific event within the $\rho$-skeleton which forces geodesic coalescence at a certain scale and to use fine properties on branching process to ensure that there are infinitely many scales at which we have geodesic coalescence.

Finally, we combine these two ingredients in Section \ref{sec:skelinfinity} where we define and study the $\delta$-skeleton of the UIPT, which we call the skeleton seen from infinity. 

\bigskip

\noindent \textbf{Acknowledgments:} We thank J\'er\'emie Bouttier for useful comments. This work was supported by the grants ANR-15-CE40-0013 (ANR Liouville),
ANR-14-CE25-0014 (ANR GRAAL), the ERC GeoBrown and the Labex MME-DII (ANR11-LBX-0023-01).

\tableofcontents

\section{Triangulations and their skeleton decompositions} \label{sec:skeletongeneral}
In this section we recall the background on combinatorics and skeleton decomposition of triangulations. For more details we refer to \cite{CLGfpp} or \cite{M}. Compared to previous works our main contribution is to perform a complete description of the skeleton decomposition in the case of \emph{finite} pointed triangulations by describing the extreme layer ($\rho$-cap and $\delta$-cap respectively in the case of $\rho$ and $\delta$-skeleton decomposition). As the reader will see, the $\delta$-skeleton decomposition is combinatorially a little bit simpler than the $\rho$-skeleton decomposition.

\subsection{Triangulations of the $p$-gon and the root transform}
As usual, all our maps will be rooted, i.e.~given with a distinguished oriented edge whose origin vertex $\rho$ is called the origin of the map. We will only consider rooted type I triangulations in the terminology of Angel \& Schramm \cite{AS}, meaning that loops and multiple edges are allowed. We will also deal with triangulations with simple boundary, which are rooted planar maps such that every face is a triangle except for the face incident to the right of the root edge which can be any simple polygon. If the length of the boundary face is $p$, we will speak of triangulations of the $p$-gon. We will use the classical convention that the map reduced to a single edge joining two distinct vertices is a triangulation of the $2-$gon, called the edge-triangulation.

One of the advantages of dealing with type I triangulations is that triangulations of the sphere can be seen as triangulations of the $1-$gon as already mentioned in \cite{CLGfpp}. To see that, split the root edge of any triangulation into a double edge and then add a loop inside the region bounded by the new double edge, at the corner corresponding to the origin of the root edge. When rerooting this new map at the loop oriented clockwise (so that the interior of the loop lies on its right) we obtain a triangulation of the $1-$gon. Note that this construction also works if the root itself is a loop, and that it is a bijection between triangulations of the sphere and triangulations of the $1-$gon. See Figure \ref{fig:transform-root}. \begin{figure}[!h]
 \begin{center}
 \includegraphics[width=1\linewidth]{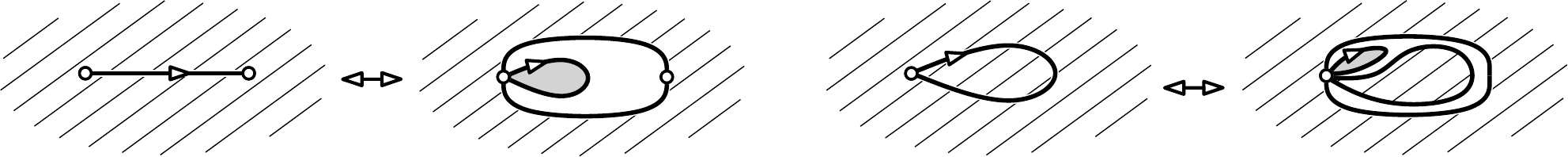}
 \caption{ \label{fig:transform-root} The root transformation (on the right when the initial root edge is already a loop).}
 \end{center} \end{figure}
\subsection{Triangulations of the cylinder and of the cone}
\label{sec:skel}

We now recall and extend the notion of triangulations of the cylinder of  \cite{CLGfpp}. Notice that our definition contains two root edges whereas that of \cite{CLGfpp} only uses one root edge.
\begin{definition} 	\label{def:trigcylinder}
Let $r \geq 0$ be an integer. A \emph{triangulation of the cylinder of height $r$} is a planar map such that all faces are triangles except for two distinguished faces called the bottom and top faces verifying:
\begin{enumerate}
\item The boundaries of the two distinguished faces form two simple cycles, called the bottom and top cycles,
\item Both boundaries contain an oriented edge (the bottom and top roots), so that the bottom face lies on the right of the bottom root and the top face lies on the left of the top root,
\item Every vertex of the top face is at graph distance exactly $r$ from the boundary of the bottom face. Furthermore when $r \geq 1$, edges of the boundary of the top face also belong to a triangle whose third vertex is at distance $r-1$ from the root face.
\end{enumerate}
For every integers $r\geq 0$ and $ p,q \geq 1$, a \emph{triangulation of the $(r,p,q)$-cylinder} is a triangulation of the cylinder of height $r$ such that its bottom face has degree $p$ and its top face has degree $q$.
\end{definition}

\begin{figure}[h!]
\begin{center}
\includegraphics[width=0.8\textwidth]{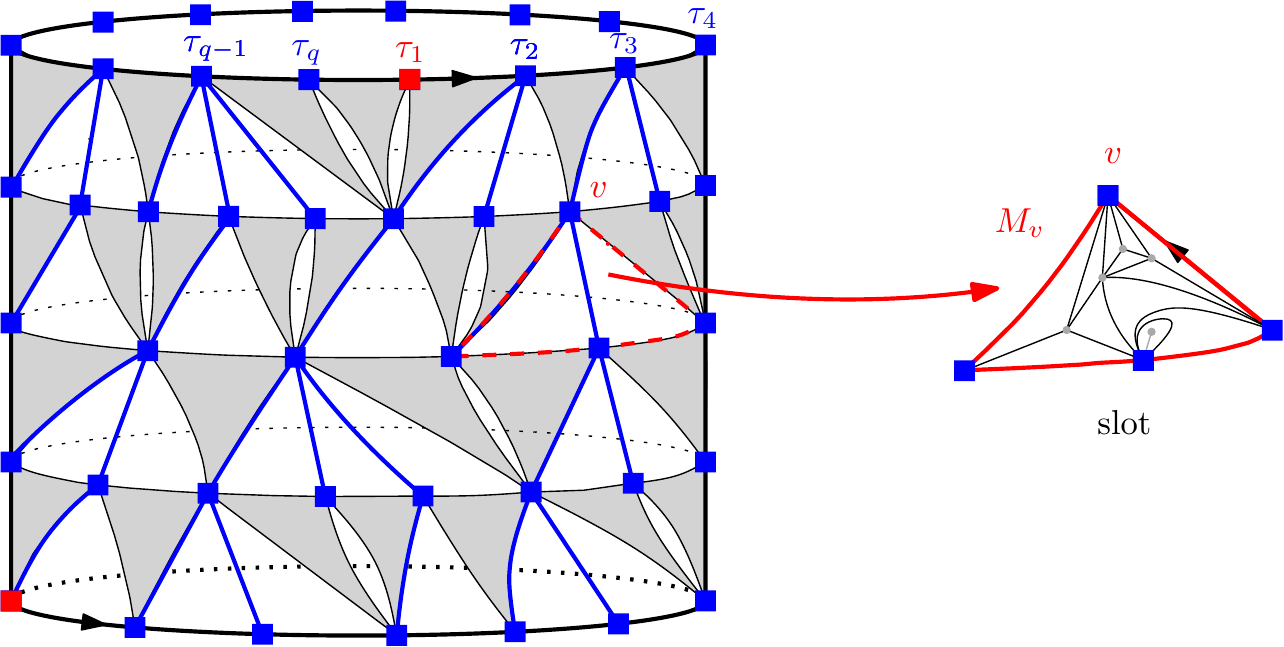}
\caption{\label{fig:genealogy} Skeleton decomposition of a triangulation of the cylinder. The distinguished vertex corresponding to the root edge of the triangulation is the red one on the bottom left. Left: construction of the forest. Right: triangulation with a boundary filling a slot.}
\end{center}
\end{figure}

The \emph{skeleton decomposition} is an encoding of triangulations of the cylinder by forests and simpler maps. In fact, describing the geometry of triangulations of the cylinder makes the bijection trivial as explained in \cites{Kr,CLGfpp}.

Fix $r \geq 0$ and $p,q\geq1$ and $\Delta$ a triangulation of the $(r,p,q)-$cylinder. First, define the growing sequence of hulls of $\Delta$ as follows: for $1 \leq j \leq r$, the ball $B_j(\Delta)$ is the union of all faces of $\Delta$ having a vertex at distance strictly smaller than $j$ from the bottom face, and the hull $\overline{B}_j (\Delta)$ consists of $B_j(\Delta)$ and all the connected components of its complement in $\Delta$ except the one containing the top face.  For every $j$, the hull $\overline{B}_j (\Delta)$ is a triangulation of the $(j,p,q')$-cylinder for some non negative integer $q'$, and we denote its top boundary by $\partial_j \Delta$. By convention $\partial_0 \Delta$ is the boundary of the bottom face of $\Delta$.

We can then define a genealogy on the vertices of $\partial_i \Delta$ for $0 \leq i \leq r$ as follows.  
Notice first that, for $1 \leq i \leq r$, every edge of $\partial_i \Delta$ belongs to exactly one face of $\Delta$ whose third vertex belongs to $\partial_{i-1} \Delta$. Such faces are called down triangles (in gray on Figure \ref{fig:genealogy}). These faces separate the map into triangulations with simple boundaries called \emph{slots} and each vertex $v \in \partial_i \Delta$ for $1 \leq i \leq r$ is located at the top of one slot $M_{v}$ (in white in Figure \ref{fig:genealogy}). We then declare that $v$ is the parent of the vertices of $\partial _{i-1} \Delta$ which belong to the bottom part of the slot $M_{v}$ except for the right-most one. This definition should be clear on 
Figure \ref{fig:genealogy}.

These relations define a forest $F$ of $q$ rooted plane trees such that the maximum height of this forest is equal to $r$. The bottom root on $\partial_{0} \Delta$ enables to distinguish a point at maximum height in this forest and the top root enables us to distinguish a tree $\tau_{1}$ and we order the other ones $\tau_{2}, ... , \tau_{q}$ in cyclic order. Together with the data given by the triangulations with a boundary filling in the slots, this is sufficient to completely describe $\Delta$.

If $F = (\tau_{1}, ... , \tau_{q})$ is a forest, we denote by $F^{\star}$ the set of all of its nodes which are not at maximum height, and if $v \in F$ we write $c(v)$ for the number of children of $v$ in $F$:

\begin{center}
\fbox{ \begin{minipage}{14cm}
\textbf{Skeleton decomposition of triangulations of the cylinder:} There is a bijection between, on the one hand triangulations of the $(r,p,q)$-cylinder $\Delta$ with $p,q,r \geq 1$ and, on the other hand, triplets $$\Big(F=(\tau_{1}, ... , \tau_{q}),x, (M_{v}:{v \in F^{*}})\Big)$$ where $F=(\tau_{1}, ... , \tau_{q})$ is a forest of $q$ trees of maximum height $r$ having $p$ vertices at maximum height, $x \in F\backslash F^{\star}$ is a distinguished vertex at maximum height and $M_{v}$ is a triangulation of the $(c(v)+2)$-gon for all $v \in F^{\star}$.
\end{minipage}}
\end{center}

\bigskip

\begin{remark} \label{rek:root}
When $p=1$ there is a unique vertex $x$ at maximum height in $F$ and we sometimes suppose as in \cite{CLGfpp} that the bottom and top roots are chosen so that $x \in \tau_{1}$. In other words, one can get rid of the top root in the case $p=1$ by declaring in the above bijection that $x \in F^\star \cap \tau_{1}$. Equivalently, one can root the top boundary uniformly.
\end{remark}

\paragraph{Extention to $p=0$ and triangulations of the cone.} The above construction can easily be extended when $p=0$, i.e.~when the bottom face is actually replaced by a single point: we speak of \emph{triangulations of the cone}. The definition is mutatis mutandis the same as Definition \ref{def:trigcylinder} (no need of  bottom root in this case), with the convention that the only triangulation of the cone with height $r=0$ is the ``vertex map'' for which $p=q=0$.  The skeleton decomposition yields in this case a bijection between triangulations of the cone of height $r$ so that the top face has perimeter $q$ and pairs 
$$\big( F= (\tau_{1}, ... , \tau_{q}) , (M_{v} :{v \in F})\big)$$ where $F$ is a forest of $q$ trees of maximum height $r-1$ and $(M_{v})$ are triangulations of the $(c(v)+2)$-gon. Notice that the slots are now indexed by the vertices of $F$ (as opposed to  $F^{\star}$ in the case of a triangulation of the cylinder). See Figure \ref{fig:cone} for an illustration. In the case of the vertex map, the forest $F$ is the empty forest. 

\begin{figure}[h!]
\begin{center}
\includegraphics[width=0.8\textwidth]{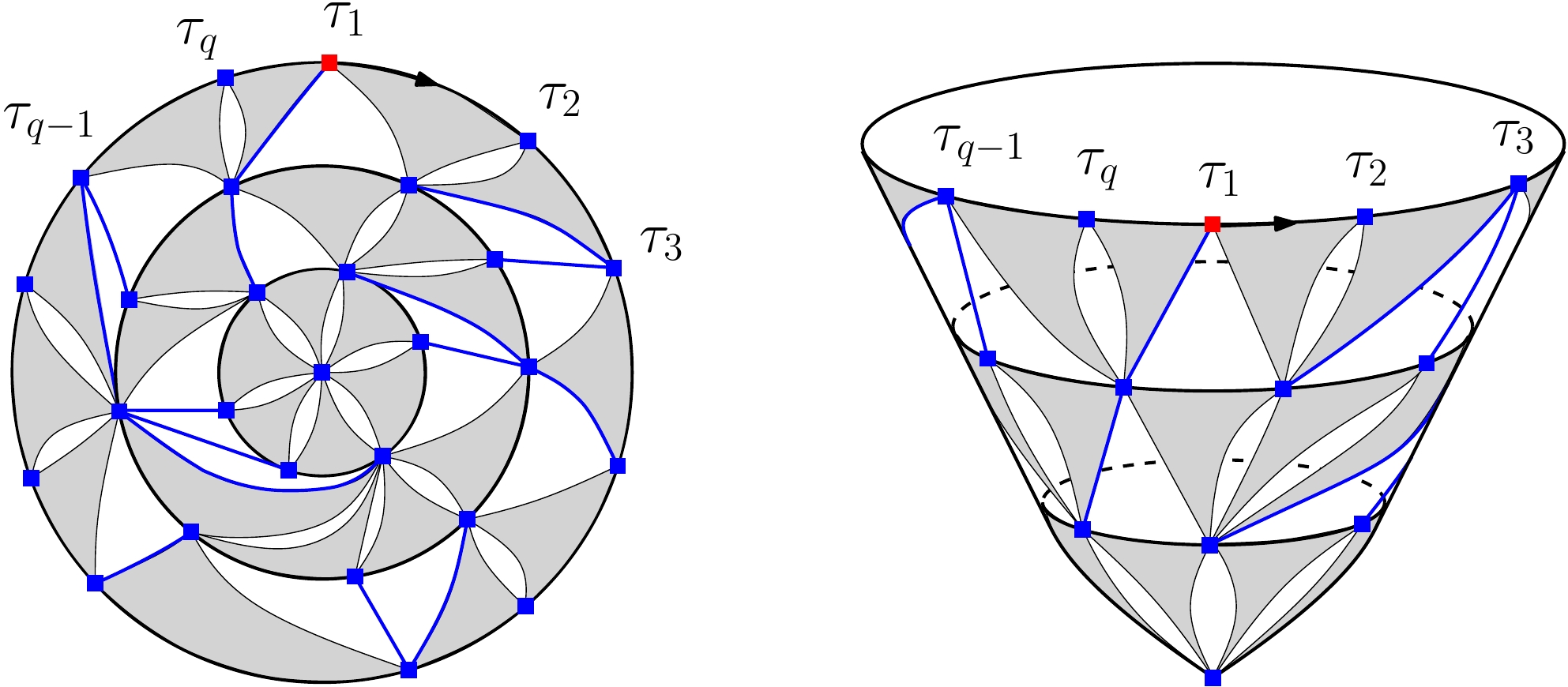}
\caption{\label{fig:cone} Skeleton decomposition of a triangulation of the cone, illustrated in two ways.}
\end{center}
\end{figure}

\subsection{The $\rho$-skeleton decomposition}

Given a rooted and pointed triangulation we will see how to associate to it triangulations of the cylinder obtained by considering hulls of balls around the origin $\rho$ of the map.  This has already been observed and used in \cites{Kr,CLGfpp,M}. This allows to completely describe finite pointed triangulations via a skeleton decomposition, except that the last layer has to be treated separately. Since the distances are measured from the origin $\rho$ of the map we call this procedure the $\rho$-skeleton decomposition.

\subsubsection{Hulls as triangulations of the cylinder}
\label{sec:hulls}
Let $( \mathfrak{t},\delta)$ be a finite rooted triangulation of the sphere given with a distinguished point $\delta$. Recall that $ \mathfrak{t}$ may also be seen as a triangulation of the $1$-gon after performing the root transform and we shall do so. We denote by $h$ the \emph{height} of $\delta$, that is the  distance between $\delta$ the origin vertex $\rho$ of $ \mathfrak{t}$. For every $r \geq 0$, the ball 
$B_r( \mathfrak{t})$ is the submap of $ \mathfrak{t}$ composed by the faces of $ \mathfrak{t}$ having at least one vertex at distance strictly smaller than $r$ from $\rho$, with the convention that $B_0( \mathfrak{t})$ is just the root edge of $ \mathfrak{t}$. For every $0 \leq r < h$ we define the \emph{hull} $\overline{B}_r ( \mathfrak{t}, \delta)$  of radius $r$ of $( \mathfrak{t}, \delta)$ by the submap of $ \mathfrak{t}$ induced by $B_r( \mathfrak{t})$ and the connected components of  $ \mathfrak{t} \setminus B_r( \mathfrak{t})$ that do not contain $\delta$. See Figure \ref{fig:hull} for an illustration.
\begin{figure}[ht!]
\begin{center}
\includegraphics[width=0.8\textwidth]{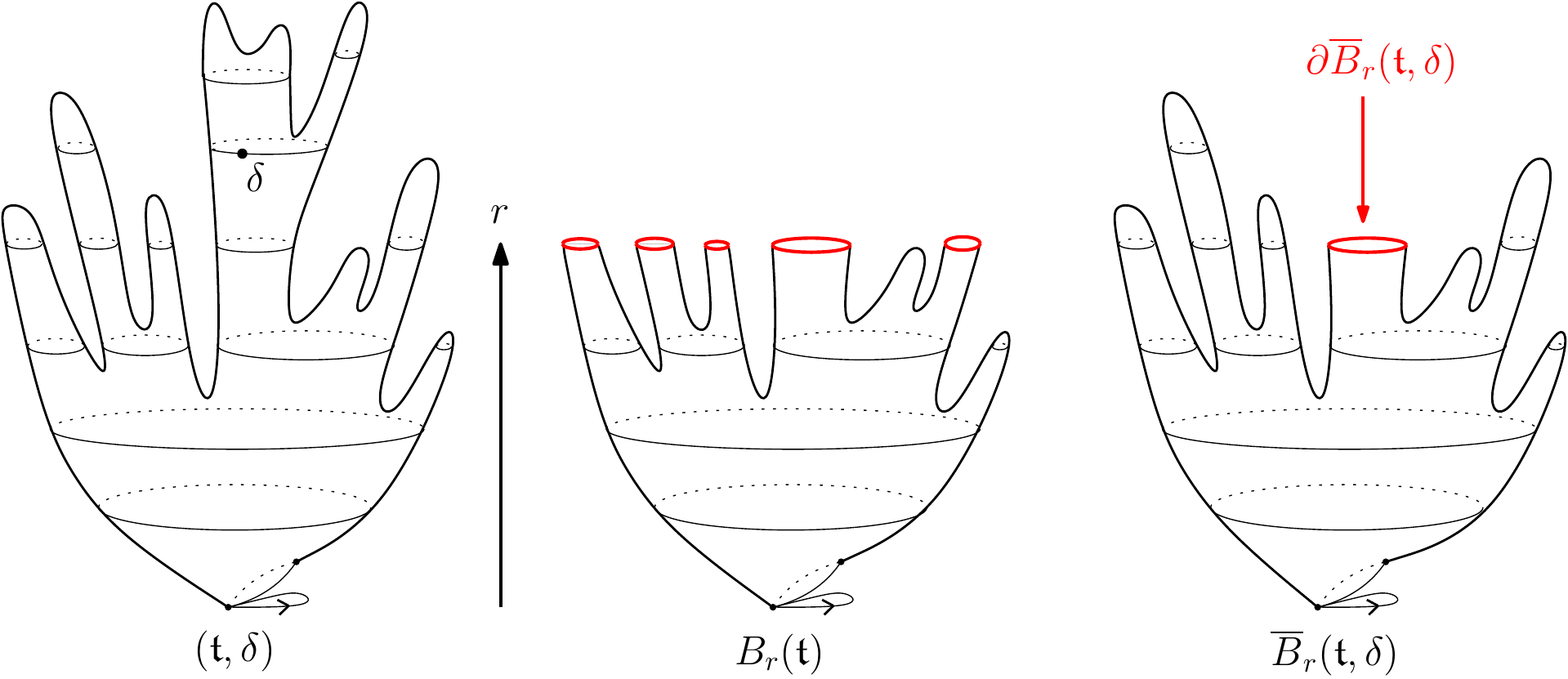}
\caption{\label{fig:hull} Ball and hull of radius $r$ in a finite triangulation. The triangulation is represented as a two-dimensional "cactus" where the height of each point is equal to its distance to the origin vertex.}
\end{center}
\end{figure}

For  any $r \geq 0$ we can define $ \overline{B}_{r}( \mathfrak{t}, \infty)$  when $ \mathfrak{t}$ is an infinite triangulation of the plane (i.e.~with one end) and $\delta= \infty$ is an imaginary point at infinity. In both cases in is easy to see that:

\begin{center} \emph{For any $r < h$ the hulls $\overline{B}_{r}( \mathfrak{t}, \delta)$ are triangulations of the $(r,1,q)$ cylinder for some $q \geq 1$}. \end{center}
Here, we implicitly used Remark \ref{rek:root} to canonically root $\overline{B}_{r}( \mathfrak{t}, \delta)$ using the root edge of $ \mathfrak{t}$ to obtain a triangulation of the cylinder whose top root starts from the unique tree of maximum height.  When dealing with the UIPT  we simply write $ \overline{B}_{r} \equiv \overline{B}_{r}( \mathbf{T}_{\infty})$ to lighten notation. In particular, we can encode $\overline{B}_{r}( \mathfrak{t},\delta)$ via the skeleton decomposition as in Section \ref{sec:skel} and we call this encoding the $\rho$-skeleton decomposition of $( \mathfrak{t}, \delta)$ at height $r$. Notice that in this encoding, the trees describing $\overline{B}_{r}( \mathfrak{t},\delta)$ grow \emph{downwards} towards the origin in $ \mathfrak{t}$.

The entire pointed triangulation is then completely described by gluing $\overline{B}_{h-1}( \mathfrak{t}, \delta)$ and the last layer $(\mathfrak{t},\delta) \setminus \overline{B}_{h-1}( \mathfrak{t}, \delta)$. This last layer is simply a triangulation of the $p-$gon having a pointed vertex $\delta$ at distance $1$ from its boundary. We call such triangulations $\rho-$caps and study them in the next subsection. This decomposition is still valid if $ \mathrm{d_{gr}}(\rho,\delta)=1$, in which case the perimeter of the $\rho$-cap is $1$ and the hull $\overline{B}_{0}( \mathfrak{t}, \delta)$ is just the root loop.

Up to adding an additional vertex corresponding to this $\rho-$cap and linking it to the ancestors of the trees describing $\overline{B}_{h-1}( \mathfrak{t}, \delta)$ we can summarize  our discussion as follows:

\begin{center}
\fbox{ \begin{minipage}{14cm}
\textbf{The $\rho$-skeleton decomposition:} The above decomposition is a bijection between, on the one hand finite rooted and pointed triangulations $ ( \mathfrak{t}, \delta)$ for which $ \mathrm{d_{gr}}(\rho,\delta)\geq 1$ and, on the other hand,  plane trees $ \mathsf{Krik} \subset \mathfrak{t}$ with a single vertex of maximal height, whose other vertices carry triangulations $(M_{v} : v \in \mathsf{{Krik}}^\star)$ where
\begin{itemize}
\item for any $v \ne \rho$ the map $M_{v}$ is a triangulation of a $(c(v)+2)$-gon,
\item for $v=\rho$ the map $M_{\rho}$ is a $\rho-$cap with perimeter $c(\rho)$.
\end{itemize}
\end{minipage}}
\end{center}

\subsubsection{The last $\rho-$layer is a $\rho-$cap.}
\label{sec:rhocap}

Let us have a more precise look at the last layer $(\mathfrak{t},\delta) \setminus \overline{B}_{h-1}( \mathfrak{t}, \delta)$. As already stated, this slice is a triangulation of the $p-$gon having a pointed vertex $\delta$ at distance $1$ from its boundary.

\begin{figure}[ht!]
\begin{center}
\includegraphics[width=0.8\textwidth]{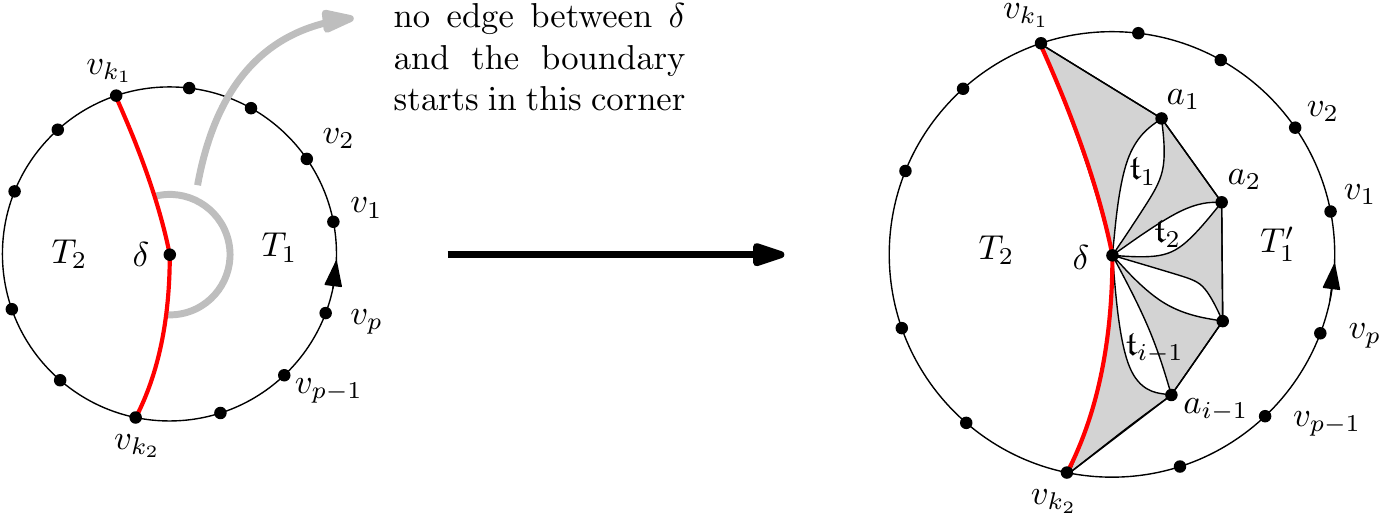}
\caption{\label{fig:rhocap} Decomposition of a $\rho$-cap.}
\end{center}
\end{figure}

Fix $\mathbb K$ a $\rho$-cap of perimeter $p$ and denote by $v_1, v_2, \ldots v_p$ the vertices of the boundary of $\mathbb K$ listed in counterclockwise order, starting with the target of the root edge $(v_p,v_1)$. Figure \ref{fig:rhocap} illustrates the following arguments. If we cut along the first edge joining $\delta$ and the boundary after the root and along the last such edge, we divide the boundary of $\mathbb K$ into two arcs. The first one is $(v_{k_2}, \ldots , v_{k_1})$ with $1 \leq k_1,\leq k_2 \leq p$ and contains the root edge. The second one is $(v_{k_1},\ldots , v_{k_2})$ and may consist of a single vertex if $k_1 = k_2$. This in turn separates the map $\mathbb K$ into two triangulations with simple boundary.

The first one, denoted $T_1$ in Figure \ref{fig:rhocap}, contains the root edge, has perimeter $p-(k_2-k_1)+2 = k + 2$ with $1 \leq k \leq p$, and has no edge between $\delta$ and $(v_{k_2}, \ldots , v_{k_1})$ other than the two edges of its boundary. The second one, denoted $T_2$, has perimeter $p-k+2$. Observe that to recover $\mathbb K$ from $T_1$ and $T_2$, we need to know where to place the root of $\mathbb K$ which is given by an additional integer $1 \leq j \leq k$.

We now consider the hull of the ball of radius $1$ around $\delta$ in $T_1$. The boundary of this hull is composed of the two boundary edges $(v_{k_1} , \delta)$ and $(\delta, v_{k_2})$ as well as a simple inner boundary whose vertices are denoted $a_{1}, a_{2}, ... , a_{i} = v_{k_2}$ in Figure \ref{fig:rhocap}. Because of our assumption, these vertices are all disjoint from $(v_{k_2}, \ldots, v_{k_1})$. Performing the skeleton decomposition, this hull can be decomposed via downward triangles as well as slots of perimeter $2$  between them which we are denoted $ \mathfrak{t}_{1}, ... , \mathfrak{t}_{i-1}$, for some $i \geq 1$ on Figure \ref{fig:crown}. Notice that there is no slot on the left of the first downward triangle $ (v_{k_1}, a_{1}, \delta)$ and on the right of the last downward triangle $ (v_{k_2}, a_{i-1}, \delta)$ due to our assumption on edges joining $\delta$ to the boundary of $T_1$. This decomposes the triangulation $T_1$ into a triangulation of the $(k+i)$-gon $T_1'$ and $i-1$ triangulations of the $2$-gon. Let us summarize: \medskip 

\begin{center}
\fbox{
\begin{minipage}{14cm}
\textbf{Decomposition of $\rho-$caps:} The above decomposition is a bijection between, on the one hand  $\rho-$caps with perimeter $p \geq 1$, and on the other hand vectors
$$ (k,j,i, ( \mathfrak{t}_{1}, ... , \mathfrak{t}_{i-1}), T_{k+i}, T_{p-k+2})$$ where $1 \leq j \leq k \leq p$  and $i\geq 1$ are integers, $ \mathfrak{t}_{1}, ... , \mathfrak{t}_{i-1}$ are triangulations of the $2$-gon, $T_{k+i}$ is a triangulation of the $(k+i)$-gon and $T_{p-k+2}$ is a triangulation of the $(p-k+2)$-gon.
\end{minipage}
}
\end{center}

\subsection{The $\delta$-skeleton decomposition}
\label{sec:skelretourne}

The $\rho$-skeleton decomposition of pointed triangulations has three main downsides. First, the tree coding the skeleton grows from $\delta$ down to the root.  Second, this tree always has a single vertex of maximal height. And third, the description of $\rho$-cap is somehow complicated. Exchanging the roles of $\delta$ and $\rho$ will not only make the tree grow from the origin vertex $\rho$ of the triangulation, but as we will see it also gives trees with no constraint on the number of vertices with maximal height. Besides, the associated notion of $\delta$-cap is simpler than the notion of $\rho$-cap. We present in this section this new decomposition that we call the $\delta$-skeleton decomposition.

\subsubsection{(Co-)Horohulls as triangulations of the cone}

As above let $( \mathfrak{t}, \delta)$ be a finite rooted and pointed triangulation of the sphere and let $h = \mathrm{d_{gr}}(\rho, \delta)$. We suppose that $h \geq 1$. By reversing the roles of $\delta$ and $\rho$, for any $r \in \{1, \ldots , h \}$, one can consider the ball of radius $h-r$ around $\delta$ in $\mathfrak t$ and denote $ \overline{\mathfrak{C}}_{r}( \mathfrak{t}, \delta)$ its hull by filling-in the connected components of its complement that do not contain the origin $\rho$ of the map and call it the \emph{co-horohull} of height $r$. We will in fact be more interested in the complement $$ \overline{H}_{r}( \mathfrak{t},\delta):= \mathfrak{t} \backslash \overline{\mathfrak{C}_{r}}( \mathfrak{t}, \delta),$$ which we name the horohull of radius $r$, see Figure \ref{fig:horohull}. When $h=r$ we simply set $ \overline{\mathfrak{C}}_h(\mathfrak t, \delta) = \delta$ so that $\overline{H}_{r}( \mathfrak{t},\delta)$ is the full triangulation $ \mathfrak{t}$ ``minus the vertex'' $\delta$ which we interpret as being $(  \mathfrak{t}, \delta)$ itself. 

\begin{figure}[ht!]
\begin{center}
\includegraphics[width=0.8\linewidth]{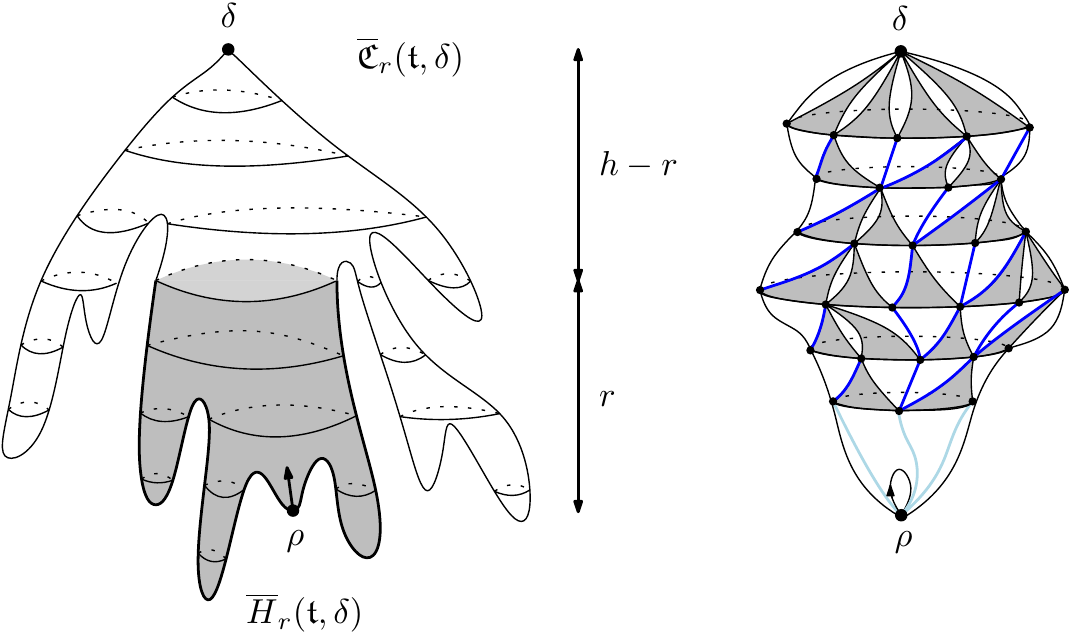}
\caption{\label{fig:horohull} Co-horohull and horohull of radius $r$ in a finite pointed triangulation. Right: When $r=1$,
the $\delta$-skeleton decomposition of $\overline{ \mathfrak{C}}_{1}( \mathfrak{t},\delta)$ (in dark blue) can be transformed into a tree by adding the light blue edges.}
\end{center}
\end{figure}

The reader may have already guessed that $ \overline{ \mathfrak{C}}_{r}( \mathfrak{t},\delta)$ is a triangulation of the cone... strictly speaking this is not the case since, at least when $r\ne h$, the object $ \overline{ \mathfrak{C}}_{r}( \mathfrak{t},\delta)$ does not possess a root edge on its boundary. To do so we choose a deterministic procedure \footnote{For example, we can consider the right-most geodesic going from $\rho$ to $\delta$ and leaving the root loop on its right and root $ \overline{ \mathfrak{C}}_{r}( \mathfrak{t},\delta)$ so that the origin of the root edge lies on this geodesic.}, depending on $ \overline{H}_{r}( \mathfrak{t}, \delta)$ only, which associates a root edge on the boundary of  $\overline{ \mathfrak{C}}_{r}( \mathfrak{t},\delta)$. Once this problem has been resolved we indeed have 
\begin{center}  \textit{For all $1 \leq r \leq h$ the map $ \overline{\mathfrak{C}_{r}}( \mathfrak{t},\delta)$ is a triangulation of the cone of height $h-r$}.  \end{center}

In particular, the entire pointed triangulation $(\mathfrak{t},\delta)$ can be described by gluing $\overline{H}_{1}( \mathfrak{t}, \delta)$ onto the triangulation of the cone $ \overline{\mathfrak{C}}_{1}( \mathfrak{t}, \delta)$ of height $h-1$. Notice that $\overline{H}_{1}( \mathfrak{t}, \delta)$ is not a triangulation of the $p$-gon in the usual sense since it carries the root edge in the bulk. They are also different from the $\rho$-caps studied in the previous section. We will study them precisely in the next subsection and name them $\delta$-caps. Up to adding an additional vertex corresponding to this $\delta$-cap and linking its edges to the roots of the trees describing $\overline{\mathfrak{C}}_{h-1}( \mathfrak{t}, \delta)$ we can summarize  our discussion as follows, see Figure \ref{fig:horohull} (right):

\begin{center}
\fbox{ \begin{minipage}{14cm}
\textbf{The $\delta$-skeleton decomposition:} The above decomposition is a bijection between, on the one hand finite rooted and pointed triangulations $ ( \mathfrak{t}, \delta)$ for which $ \mathrm{d_{gr}}(\rho,\delta)\geq 1$ and, on the other hand,  plane trees $ \mathsf{Skel} \subset \mathfrak{t}$ whose vertices carry triangulations $(M_{v} : v \in \mathsf{{Skel}})$ where
\begin{itemize}
\item for any $v \ne \rho$ the map $M_{v}$ is a triangulation of a $(c(v)+2)$-gon,
\item for $v=\rho$ the map $M_{\rho}$ is a horohull of radius $1$ with perimeter $c(\rho)$.
\end{itemize}
\end{minipage}}
\end{center}

\begin{remark} When $ \mathrm{d_{gr}}(\rho,\delta)=1$ then $\overline{ \mathfrak{C}}_{1}( \mathfrak{t},\delta)$ is reduced to the vertex map and the $\delta$-cap $\overline{H}_{1}( \mathfrak{t}, \delta)$ is equal to $( \mathfrak{t},\delta)$ itself. To be coherent, we say that in this case $\overline{H}_{1}( \mathfrak{t}, \delta)$ is a $\delta$-cap with perimeter $0$, so that the tree $ \mathsf{Skel}$ consists of a single vertex.
\end{remark}

\subsubsection{The first $\delta-$layer is a $\delta-$cap}
\label{sec:caps}

Let us have a more precise look at the structure of $\overline{{H}}_1 ( \mathfrak{t}, \delta)$, the horohull of radius $1$ constituting the first layer of the $\delta$-skeleton decomposition. By the root transform, this is a planar map of the $1-$gon.  If $ \mathrm{d_{gr}}(\rho,\delta) \geq 1$ all of its interior faces are triangles except for one face which is bounded by the simple cycle $\partial \overline{H}_{1}( \mathfrak{t},\delta)$ of perimeter $p \geq 1$. Furthermore, this special face has at least one vertex at distance exactly $1$ from $\rho$. We call such triangulations $\delta$-cap, and the degree of the special face is the perimeter of the $\delta$-cap. When $\mathrm{d_{gr}}(\rho,\delta) = 1$ we set $p=0$ and declare that a $\delta$-cap with perimeter $0$ is simply a rooted pointed triangulation $( \mathfrak{t}, \delta)$ such that $ \mathrm{d_{gr}}(\rho,\delta)=1$. In this case the distinguished vertex $\delta$ should be interpreted as a ``boundary of perimeter $0$''. \medskip 

Fix $ \mathbb{K}$ a $\delta$-cap with perimeter $p$, Figure \ref{fig:crown} illustrates the following arguments. We cut along the first edge (in clockwise order) before the root loop joining $\rho$ and the outer boundary of $\mathbb{K}$, as well as along the root loop to get a triangulation of the $(p+3)$-gon $\mathbb{K}'$. Let us denote by $\rho_1, \rho_2, v_1, \ldots , v_p,v'_1$ the vertices of the boundary listed in counterclockwise order, the oriented edge $(\rho_1,\rho_2)$ being the root edge of $\mathbb{K}'$ (corresponding to the root-loop of $\mathbb{K}$) and $v_1, v_1'$ being the same vertex in $\mathbb{K}$. Since we took the first edge joining $\rho$ to the outer boundary of $\mathbb{K}$, the triangulation $\mathbb{K}'$ has no edge joining $\rho_1$ and the vertices $v_1, \ldots, v_p$. In addition, the boundary edge $(v_1',\rho_1)$ is the only edge of $\mathbb{K}'$ joining the two vertices $v_1'$ and $\rho_1$. This operation is still valid if $p=0$ and the vertices $v_1$ and $v'_1$ are the same.

\begin{figure}[ht!]
\begin{center}
\includegraphics[width=0.8\textwidth]{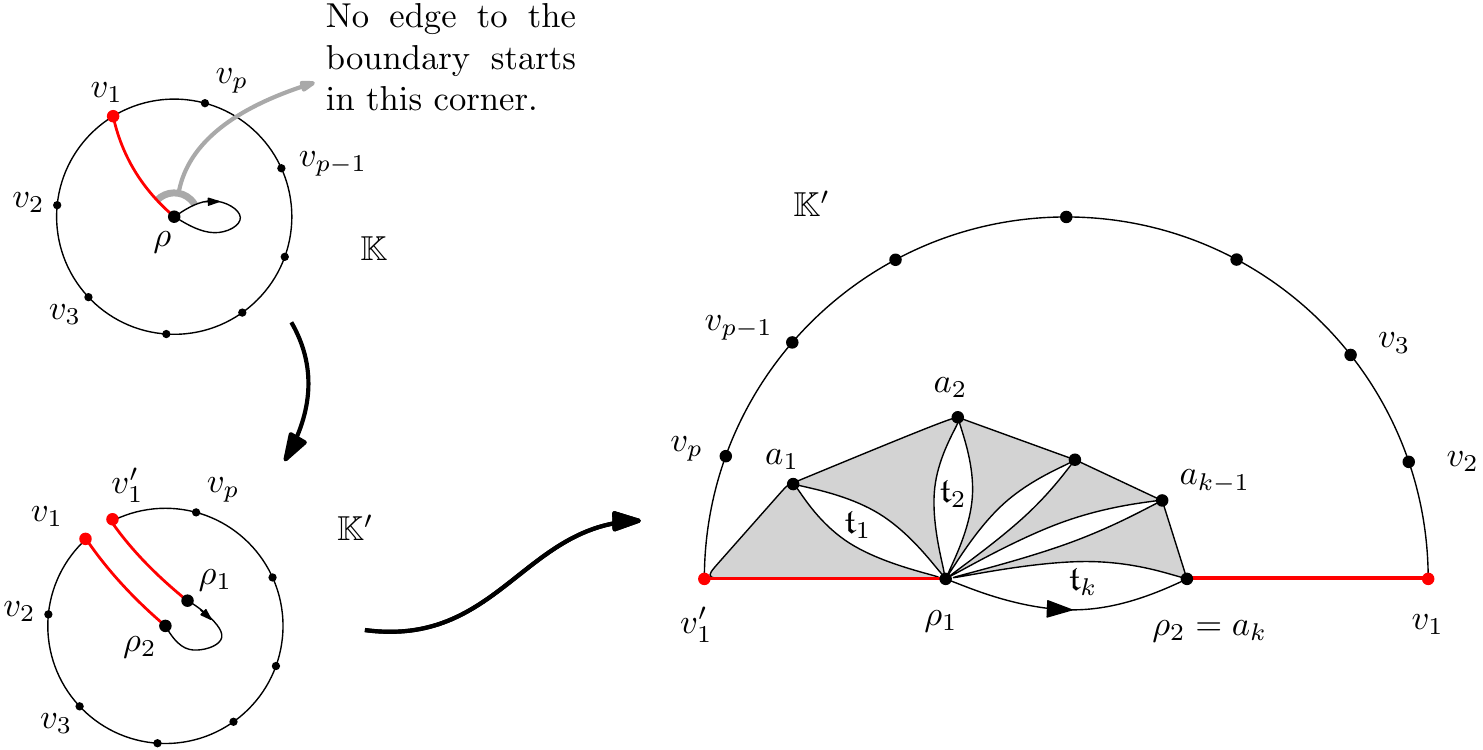}
\caption{\label{fig:crown} Decomposition of a $\delta$-cap.}
\end{center}
\end{figure}

We now consider the hull of the ball of radius $1$ around $\rho_{1}$ in $\mathbb{K}'$. The boundary of this hull is composed of the two boundary edges $(v_{1}' , \rho_{1})$ as well as $(\rho_{1}Ê, \rho_{2})$ and of a simple inner boundary whose vertices are denoted $a_{1}, a_{2}, ... , a_{k} = \rho_{2}$ as in Figure \ref{fig:crown}. Because of our assumption, these vertices are all disjoint from $v_{1}, v_{2}, ... , v_{p},v_{1}'$. Performing the skeleton decomposition, this hull can be decomposed via downward triangles as well as slots of perimeter $2$  between them which we are denoted $ \mathfrak{t}_{1}, ... , \mathfrak{t}_{k}$ on Figure \ref{fig:crown}. Notice that there is no slot on the left of the first downward triangle $ (v_{1}', a_{1}, \rho_{1})$ because we supposed that $(v_{1}' ,\rho_{1}) $ is the only edge joining $v_{1}'$ and $\rho_{1}$. Let us summarize: 

\begin{center}
\fbox{ \begin{minipage}{14cm}
\textbf{Decomposition of $\delta$-caps:} The above decomposition is a bijection between, on the one hand $\delta$-caps with perimeter $p \geq 0$, and on the other hand triplets
$$ (k, ( \mathfrak{t}_{1}, ... , \mathfrak{t}_{k}), T_{p+k+1})$$ where $k \geq 1$ is an integer, $ \mathfrak{t}_{1}, ... , \mathfrak{t}_{k}$ are triangulations of the $2$-gon and $T_{p+k+1}$ is a triangulation of the $(p+k+1)$-gon.
\end{minipage}}
\end{center}

\section{Generating series} \label{sec:gs}
In this section we derive the generating series of the various models of triangulations we encountered (cylinder, cone and caps), counted by vertices. In a more probabilistic point of view, these generating series are, up to multiplicative constants, the generating functions of the volume of Boltzmann triangulations (of the cylinder, the cone or caps).  To perform these explicit computations, we rely on enumeration results developed in \cite{M} which are recalled below. We then derive a few interesting distributional consequences, among which a quick derivation of the $2$ point function of 2D gravity (Proposition \ref{prop:2point}).

\subsection{Enumeration of triangulations} \label{sec:combi}
The enumeration of triangulations of the $p-$gon is now well known and can be found  in \cite{GouldenJackson}. Let $\mathcal{T}_{n,p}$ be the set of all triangulations of the $p-$gon with $n$ inner vertices (\emph{i.e.} vertices that do not belong to the boundary face) and define the bivariate generating series
\begin{equation}
\label{eq:Tseries}
T(x,y) = \sum_{p \geq 0} \sum_{n \geq 0} \left| \mathcal{T}_{n,p} \right| x^n y^{p} = \sum_{p \geq 0} T_p(x) \, y^{p}.
\end{equation}
This series can be computed explicitly and is analytic in both variables if $|x| < x_c$ and $|y| < y_c$ where
\begin{equation}
\label{eq:radii}
x_c := \frac{1}{12 \sqrt{3}}
\quad \text{and} \quad
y_c := \frac{1}{12}.
\end{equation}
We refer to \cite{Kr} for details and in particular for closed formulae for  $\# \mathcal{T}_{n,p}$. As is now classical, we call a critical Boltzmann triangulation of the $p$-gon a random triangulation of the $p$-gon $\mathsf T$ whose law is given by
\begin{align} \label{eq:Boltz}
\mathbb P \left( \mathsf T = \Delta \right) = \frac{x_c^{|\Delta|-1}}{T_p(x_c)}
\end{align}
for every triangulation of the $p$-gon $\Delta$.

A corollary of the enumeration of triangulations, which is derived in \cite[Section 3.1]{M} is the following. For $s \in [0,1]$ denote by $t := t(s) \in [0,1]$ the unique power series in $s$ solution of the equation
\begin{equation}
\label{eq:paramst}
s = t \sqrt{3-2t}
\end{equation}
such that $t(0) = 0$.
This series has nonnegative coefficients, is analytic on $\mathbb C \setminus [1,+\infty[$, and $t(1)=1$. To simplify notation, we will drop the dependency in $s$ and simply write $t$ for $t(s)$. With this notation, for every $s \in [0,1]$, define the function $\varphi_t$ for $z\in [0,1]$ by
 \begin{eqnarray}
\varphi_t(z) &\overset{ \mathrm{def}}{=}& x_c  s \sum_{p \geq 0}   (y_c  t)^{p-1} \cdot T_{p+2}(x_c  s) \, z^p\nonumber \\  &=& \frac{x_c  s}{(y_c  t)^2 \cdot z} \left( T(x_c  s, y_c  t \cdot z) - T(x_c  s,0) \right)\nonumber \\
& \underset{ \mbox{\cite[Sec. 3.1]{M}} }{=}&  1 - \left( \frac{1}{\sqrt{1-z}} \sqrt{\frac{3 -2t}{t}} + \sqrt{1 + \frac{3}{1-z} \left(\frac{1 -t}{t} \right)}\right)^{-2}. \label{eq:exprPhit} \end{eqnarray}
This is the generating function of a probability distribution $\theta_{t}$ on $\mathbb Z_+$. When $t=1$ we have $\varphi_{t} = \varphi$  as in \eqref{eq:offspring} and so $\theta_{1}$ coincides with the offspring distribution $\theta$. An important  fact \cite[Lemma 3]{M} is that the iterates of these generating functions are again explicit: if $\varphi_t^{\{r\}}$ is the $r$-th iterate of the function $\varphi_t$  then 
\begin{equation}
\label{eq:iterPhit}
\varphi_t^{\{ r\}} (z) = 1 - 3 \frac{1-t}{t \, \left( \sinh \left( \sinh^{-1} \left( \sqrt{\frac{3(1-t)}{t(1-z)}}\right) + r \cosh^{-1} \left(\sqrt{\frac{3-2t}{t}} \right) \right) \right)^2}
\end{equation}
for every $r \geq 0$ and $z \in [0,1]$. A simple computation gives the alternate expression
\[
\varphi_t^{\{r\}} (z) = 1- \frac{1-z}{
\left(
\sqrt{1+\frac{t(1-z)}{3(1-t)}} \, 
\sinh \left( r \cosh^{-1} \left(\sqrt{\frac{3-2t}{t}}\right) \right)
+
\cosh \left( r \cosh^{-1} \left(\sqrt{\frac{3-2t}{t}}\right) \right)
\right)^2},
\]
showing that the radius of convergence of $\varphi_t^{\{r\}}$ is $\frac{3-2t}{t} \geq 1$, and that these functions are analytic on $\mathbb C \setminus [\frac{3-2t}{t} , + \infty[$, making them amenable to singularity analysis and transfer theorems, a fact that we will need later.

\subsection{Triangulations of the cylinder and of the cone and the caps}

\paragraph{Cylinder.} For $r \geq 0$ and $p, q\geq 1$, we denote by $C_{r,p,q}(x)$ the generating series of triangulations of the $(r,p,q)-$cylinder counted by vertices. That is $C_{r,p,q}(x) = \sum_{\mathfrak t} x^{|\mathfrak t|}$, the sum being taken over all triangulations of the $(r,p,q)-$cylinder and $x \in [0,x_{c}]$. In the skeleton decomposition, each vertex of the coding forest, except those of maximum height, corresponds to a triangulation of the $(c(v)+2)$-gon. Recalling that for every $m \geq 1$ the generating series of triangulations of the $m$-gon counted by \emph{inner} vertices is $T_m(x)$, the skeleton decomposition of triangulations of the cylinder (Section \ref{sec:skel}) translates into the identity
\[
C_{r,p,q}(x) = p \cdot x^p \sum_{F \in \mathcal{F}(r,p,q)} \quad \prod_{v \in F^\star} x T_{c(v)+2}(x)
\] 
where $\mathcal{F}(r,p,q)$ is the set of all forests of $q$ trees with maximum height $r$ and with $p$ vertices at maximum height. The factor $p$ comes from the marked vertex on the bottom cycle and the term $x^p$ counts vertices of the bottom cycle which are not taken into account in the product over $F^\star$. In the product, each $x$ before the series $T_{c(v)+2}$ allows to count vertices of each cycle of the skeleton. Since $\prod_{v \in \tau} \alpha^{c(v) -1} = \alpha^{-1}$ for every rooted finite tree $\tau$ and every real $\alpha$, using the notation introduced in the last section and putting $x = x_{c}s$ for $s \in [0,1]$ with the corresponding $t=t(s)$, the last display can be written as 
\begin{align}
C_{r,p,q}(x_c \cdot s) &= p \cdot (x_c \cdot s)^p (y_c \cdot t)^{q-p} \sum_{F \in \mathcal{F}(r,p,q)} \quad \prod_{v \in F^\star}
\left( (y_c \cdot t )^{c(v)-1} x_c s \cdot T_{c(v) + 2}(x_c s) \right)\nonumber \\
&=p \cdot (x_c \cdot s)^p (y_c \cdot t)^{q-p} \sum_{F \in \mathcal{F}(r,p,q)} \quad \prod_{v \in F^\star}
\left[ z^{c(v)}  \right]\varphi_t(z)\nonumber \\
\label{eq:gencylinder}
 &=
p \cdot (x_c \cdot s)^p (y_c \cdot t)^{q-p}
\cdot \left[ z^p  \right] \left(\varphi_t^{\{r\}} (z) \right)^q,
\end{align}
where in the last line we used the fact that $\varphi_{t}$ is the generating function of the offspring distribution $\theta_{t}$. The last factor can thus be  interpreted as the probability that a $\theta_{t}$-Galton--Watson process started with $q$ particles has $p$ particles at generation $r$.

\paragraph{Cone.} For $r\geq 0$ and $q\geq 0$, the derivation of the generating series $C_{r,0,q}(x)$ of triangulations of the cone of height $r$ and perimeter $q$ counted by vertices is very similar. The only differences are that there is no bottom root, the forest has maximum height $r-1$ and every vertex of the forest has an associated slot (even vertices of maximum height). It is then straightforward that
\[
C_{r,0,q}(x) = x \sum_{F \in \mathcal{F}(r-1,q)} \quad \prod_{v \in F} x T_{c(v)+2}(x)
\]
where $\mathcal{F}(r-1,q)$ denotes the set of all forests of $q$ trees of maximum height $r-1$ (the first $x$ counts the origin vertex of the cone). Doing the same trick as above yields when $r \geq 1$
\begin{align}
C_{r,0,q}(x_c \cdot s)
&= (x_c \cdot s) (y_c \cdot t)^{q} \sum_{F \in \mathcal{F}(r-1,q)} \quad \prod_{v \in F}
\left[ z^{c(v)}  \right]\varphi_t(z)\nonumber \\
\label{eq:gencone}
C_{r,0,q}(x_c \cdot s) &=
(x_c \cdot s) (y_c \cdot t)^{q}
\left( \left(\varphi_t^{\{r\}} (0) \right)^q - \left(\varphi_t^{\{r-1\}} (0) \right)^q \right).
\end{align}
where the last factor can be interpreted as the probability that a $\theta_{t}$-Galton--Watson process started with $q$ particles dies out exactly at generation $r$. When $r=0$ we obviously have $C_{0,0,0}(x_c \cdot s) = x_{c}  s$.

\paragraph{Caps.} Since $\delta-$caps will be more useful for our purpose, let us deal with them first.
The decomposition of $\delta$-caps presented in Section \ref{sec:caps} easily translates into the following identity for the generating series of $\delta$-caps of perimeter $p \geq 0$ where every vertex \emph{that is not on the boundary} comes with a weight $x$ (this convention will allow to glue $\delta$-caps and cones without counting twice the vertices of the glued boundaries):
\begin{align}
\label{eq:crowngs}
K_p(x) =
\sum_{k \geq 1} T_{p+k+1} (x) \left( xT_2(x) \right)^k.
\end{align}
In particular $K_{0}(x)$ is the generating series of triangulations pointed at distance $1$ from the origin.

For later uses we will need the bivariate generating function of $\delta$-caps counting volume and perimeter. More precisely we introduce the following function
\begin{align}
K(s,z) &:= x_c s \sum_{p\geq 0} K_p(x_c s) \left(y_c t z\right)^p\label{def:K} \\
&= x_c s \sum_{n \geq 2} T_n(x_c s) (y_c t z)^{n-1} \sum_{k=1}^{n-1} \left( \frac{x_c s T_2 (x_c s)}{y_c t z}\right)^k\nonumber \\
& = x_c s y_c t \varphi_t (0) \sum_{n \geq 2} T_n(x_c s) (y_c t z)^{n-2} \cdot \frac{1 - \left(\frac{\varphi_t (0)}{z}\right)^{n-1}}{1 - \frac{\varphi_t (0)}{z}} \nonumber
\end{align}
where we used the identity
\[
\varphi_t(0) = \frac{x_c s T_2(x_c s)}{y_c t}
\]
to get the last line. With the definition \eqref{eq:exprPhit} of $\varphi_t$ this gives
\begin{equation}
\label{eq:Ksz}
K(s,z) = (y_c t)^2 \frac{\varphi_t (0)}{1 - \frac{\varphi_t (0)}{z}} \left(\varphi_t(z) - \frac{\varphi_t(0)}{z} \varphi^{\{2\}}_t(0) \right).
\end{equation}
The case when $s=1$ will be of particular interest:
\[
K(1,z) = y_c^2 \left( 1-2\frac{1-z}{(1+2\sqrt{1-z})(1+\sqrt{1-z})} \right),
\]
where we recognize the generating function $\psi$ of $\nu$ given in equation \eqref{def:nu} inside the parenthesis. Note for future reference that the expectation of $\nu$ is $2$ and that for every $p \geq 0$:
\begin{equation}
\label{eq:nuval}
\nu(p) = \frac{x_c}{y_c^2} \, y_c^p \, K_p(x_c) = 4 \sqrt 3 \, y_c^p \, K_p(x_c).
\end{equation}
When $p \geq 0$ is fixed, we call a (critical) Boltzmann $\delta$-cap with perimeter $p$, a random $\delta$-cap $ \mathsf{Cap}_{\delta}(p)$ with perimeter $p$ whose law is given by 
\begin{equation} \label{eq:BolKip}
\mathbb{P}(\mathsf{Cap}_{\delta}(p) = \Delta) = \frac{x_{c}^{|\Delta|-p}}{K_{p}(x_{c})}.
\end{equation}

\bigskip

The decomposition presented in Section \ref{sec:rhocap}  could also be translated into an expression for the generating series of $\rho$-caps of perimeter $q \geq 1$ counted by total number of vertices, but since we shall not use $\rho$-caps in what follows we leave this exercise to the interested reader.

\subsection{Distributional consequences}
We now use the exact expressions for the generating functions derived above to compute a few distributions on random Boltzmann triangulations. Those computations could be made using either the $\rho$-skeleton decomposition or the $\delta$-skeleton decomposition but to keep the paper short we will concentrate only on the latter which turns out to be a little simpler since the associated trees have no constraints. Let $ \mathbf{T}^{\bullet}$ be a critical pointed triangulation, that is a triangulation of the $1$-gon pointed at an inner vertex $\delta$ whose weight is proportional to $x_{c}^{| \mathfrak{t}|}$. For any $h \geq 1$ we denote by $  \mathbf{T}^{\bullet}_{h}$  the last random triangulation conditioned on the event where the distance between the pointed vertex and the origin of the map is exactly  $h$. 

\subsubsection{The $\delta$-skeleton of $\mathbf{T}^{\bullet}$}
With the above notation, we consider $( \mathsf{Skel}, (M_{v} : v \in \mathsf{Skel} )$ the $\delta$-skeleton decomposition of $ \mathbf{T}^{\bullet}$ as in Section \ref{sec:skelretourne}. Recall the definitions of $\theta$ and $\nu$ given in the Introduction as well as the definition of critical Boltzmann $\delta$-cap $\mathsf{Cap}_{\delta}$ given in equation \eqref{eq:BolKip} and Boltzmann triangulations of the $p$-gon given in equation~\eqref{eq:Boltz}.
\begin{proposition} 
\label{prop:skelfinite}
The random tree $ \mathsf{Skel}$ is a (modified) Galton--Watson tree where all the vertices have offspring distribution $\theta$ except the ancestor whose offspring distribution is $\nu$. Conditionally on $ \mathsf{Skel}$ the slots $M_{v}$ are independent and distributed as critical Boltzmann triangulations of the $(c(v)+2)$-gon except $M_{\rho}$ which is distributed as $\mathsf{Cap}_{\delta}(c(\rho))$.
\end{proposition}
\begin{proof}
Let $(\mathfrak t , \delta)$ be a triangulation of the $1$-gon pointed at an inner vertex. By definition
\[
\mathbb P \left( \mathbf{T}^{\bullet} = (\mathfrak t , \delta) \right) = \frac{x_c^{|\mathfrak t|}}{x_c^2 \, T_1'(x_c)},
\]
where we recall that $T_1$ is the generating series of triangulations counted by inner vertices and $| \mathfrak{t}|$ denotes the total number of vertices of $ \mathfrak{t}$. A basic computation gives $T'_{1}(x_{c})=3$ so that $x_c^2 T_1'(x_c) = y_c^2$. Applying the $\delta$-skeleton decomposition to $(\mathfrak t,\delta)$ we get
\begin{align*}
\mathbb P \left( \mathbf{T}^{\bullet} = (\mathfrak t , \delta) \right) & = \left( \frac{x_c}{y_c^2} K_{c(\rho)}(x_c)\right) \left( \frac{x_c^{|\overline H_1(\mathfrak{t}, \delta)|-|\partial \overline H_1(\mathfrak{t}, \delta)|}}{K_{c(\rho)}(x_c)}\right) \left( \prod_{v \in  \mathsf{Sk} \setminus \rho} x_c^{ \mathrm{inn}(M_v)+1}\right)
\end{align*}
where $ \mathsf{Sk}$ denotes the tree coding the $\delta$-skeleton of $ (\mathfrak{t}, \delta)$ and $ \mathrm{inn}( \mathfrak{t})$ is the number of inner vertices of $ \mathfrak{t}$. The first factor $x_c$ counts the vertex $\delta$. Using the same trick as the one used to obtain equation \eqref{eq:gencylinder}, we get
\begin{align*}
\mathbb P \left( \mathbf{T}^{\bullet} = (\mathfrak t , \delta) \right) & = \left( \frac{x_c}{y_c^2} K_{c(\rho)}(x_c) \, y_c^{c(\rho)}\right) \left( \frac{x_c^{|\overline H_1(\mathfrak{t}, \delta)|-|\partial \overline H_1(\mathfrak{t}, \delta)|}}{K_{c(\rho)}(x_c)}\right) \left( \prod_{v \in  \mathsf{Sk} \setminus \rho} \frac{x_c^{ \mathrm{inn}(M_v)}}{T_{c(v)+2}(x_c)}\right) \left( \prod_{v \in \tau \setminus \rho} \theta(c(v))\right).
\end{align*}
This finishes the proof since the first parenthesis is just $\nu(c(\rho))$.
\end{proof}

As a simple consequence of this proposition, we can compute the law of the distance between the origin and the distinguished point in $ \mathbf{T}^{\bullet}$, see \cites{BDFG03,DF05}, \cite[Proposition 2.4]{CD06} as well as \cite[Section 6.1]{BG12} for analogous results (some of them in the case of quadrangulations) derived with bijective methods

\begin{corollary} Denote $H$ the distance between the origin and the distinguished point in $ \mathbf{T}^{\bullet}$. Then $H$ has the following distribution
$$ \mathbb{P}(H=k) = \frac{4}{k(k+1)(k+2)}, \quad \mbox{ for } k \geq 1. $$
\end{corollary}
\begin{proof}
Thanks to the law of the $\delta$-skeleton, the height $H$ satisfies
$$ \mathbb{P}( H \leq k+1) = \mathbb{E}[q_{k}^{\partial \overline{H}_{1}}],$$
where $q_{k}$ is the probability that a $\theta$-Galton Watson tree has height smaller than or equal to $k$ and $\partial \overline{H}_{1}$ is the perimeter of the cap of the triangulation. Using \eqref{eq:iterPhit} or \cite[Eq (28)]{CLGfpp} we have $q_{k} = 1-(k+1)^{-2}$. Therefore, plugging this into the generating function of $\nu$ we get:
\[
\mathbb{P}( H \leq k+1) = \psi(q_k) = 1 - 2 \frac{1}{(k+2)(k+3)}
\]
giving the result.
\end{proof}

\subsubsection{Two point function of 2D quantum gravity}
\label{sec:2point}
In the same vein as what we did above, we can use the $\delta$-skeleton decomposition (Section \ref{sec:skelretourne}) to compute the generating series $G_h$ of pointed triangulations $T^{\bullet}_{h}$ counted by vertices. Such series are usually called (discrete) two point functions in the physics literature (see the works of Bouttier, Di Francesco and Guitter \cites{BDFG03,BG12,G17} for several alternate computations of the same function).

For $h \geq 0$, to obtain a triangulation pointed at height $h + 1 \geq 1$ we simply glue together a triangulation of the cone of height $h$ and a cap with the same perimeter $p \geq 0$. This gives the following identity: 
\begin{align}
G_{h+1}(x_c \cdot s) &= \sum_{p\geq 0} C_{h,0,p}(x_c \cdot s) \, K_p(x_c \cdot s)\nonumber \\
& \underset{ \eqref{eq:gencone}}{=}
(x_c \cdot s) \sum_{p\geq 0} K_p(x_c \cdot s)  (y_c \cdot t)^{p} \left(\varphi_t^{\{h\}} (0) \right)^p
- 
(x_c \cdot s) \sum_{p\geq 0} K_p(x_c \cdot s)
 (y_c \cdot t)^{p} \left(\varphi_t^{\{h-1\}} (0) \right)^p
\nonumber \\
&\underset{ \eqref{def:K}}{=} \left( K(s, \varphi_t^{\{h\}}(0)) -  K(s, \varphi_t^{\{h-1\}}(0)) \right)\nonumber \\
& = ( y_c t )^2 \varphi_t(0) \left(
\frac{
\varphi_t^{\{h+1\} }(0)
- \varphi^{\{2\}}_t(0)
}
{
1-
\frac{\varphi_t(0)}{\varphi_t^{\{h\}}(0)}
}
-
\frac{
\varphi_t^{\{h\}}(0)
- \varphi^{\{2\}}_t(0)
}
{
1-
\frac{\varphi_t(0)}{\varphi_t^{\{h-1\}}(0)}
}
\right). \label{eq:gcexact}
\end{align}
The fact that this function is explicit allows for very simple proofs of the Propositions \ref{prop:2point} and \ref{prop:UIPTdisth} stated in the introduction.

\begin{proof}[Proof of Proposition \ref{prop:UIPTdisth}]
A straightforward singularity analysis using  \eqref{eq:gcexact} and \eqref{eq:iterPhit} yields for every $h \geq 1$ and as $s\to1$:
\begin{align*}
G_h(x_c s) &= 
\frac{1}{36 h (h+1) (h+2)}
- \frac{1}{90} \frac{h^4+4h^3+5h^2+2h+3}{h(h+1)(h+2)} \left(1-s\right)\\
& \quad + \frac{\sqrt{6}}{2835} \frac{-1+16h+80h^2+152h^3+138h^4+60h^5+10h^6}{h(h+1)(h+2)} \left( 1-s \right)^{3/2}
+ \mathcal{O} \left( 1-s \right)^{3/2}.
\end{align*}
The reader eager to verify this computation (and others in the paper) can use the companion Maple worksheet available on the authors websites.
By a transfer theorem \cite[Theorem VI.4 p. 393]{FS}, the number of triangulations with $n$ vertices pointed at distance $h$ of the origin vertex satisfies as $n \to \infty $
\[
\left[ x^{n} \right] G_h(x) \sim \frac{3}{4 \sqrt \pi} \frac{\sqrt{6}}{2835} \frac{-1+16h+80h^2+152h^3+138h^4+60h^5+10h^6}{h(h+1)(h+2)}
(x_c)^{-n} n^{-5/2},
\]
\end{proof}

\begin{remark}
Recall that the number of triangulations of the $1$-gon with $n$ inner vertices satisfies $[ x^{n} ] T_1(x) \sim \frac{1}{6 \sqrt{2 \pi}} (x_c)^{-n} n^{-5/2}$ as $n \to \infty$. Therefore, if $ \mathcal{X}_{n,h}$ denotes the number of vertices at distance $h$ from the origin in a uniformly chosen triangulation (of the $1$-gon) with $n$ inner vertices then we have 
\begin{equation}
\label{eq:hUIPT}
\mathbb{E}[ \mathcal{X}_{n,h}] = \frac{\left[ x^{n+1} \right] G_h(x)}{\left[ x^{n} \right] T_1(x)} \xrightarrow[n\to\infty]{} \frac{4}{35} \frac{-1+16h+80h^2+152h^3+138h^4+60h^5+10h^6}{h(h+1)(h+2)}.
\end{equation}
By the local convergence towards the UIPT, we have $ \mathcal{X}_{n,h} \to \mathcal{X}_{\infty,h}$ in distribution where $\mathcal{X}_{\infty,h}$ is the number of vertices at distance $h$ from the origin in the UIPT. Therefore, it suffices to establish that the sequence of random variables $(\mathcal{X}_{n,h})_n$ is uniformly integrable to show that the expected number of vertices at distance $h$ from the origin in the UIPT is given by the limit \eqref{eq:hUIPT}. There are many ways to prove this but we refrain from doing so in this paper.
\end{remark}

\begin{proof}[Proof of Proposition \ref{prop:2point}]
The Laplace transform of the proposition can be expressed as:
$$ \mathbb{E}[ e^{-\lambda h^{-4} |T_{h}^{\bullet}|}] = \frac{G_{h}( x_{c} \cdot e^{-\lambda h^{-4}})}{G_{h}(x_{c})}.$$
Using the exact form of $G_{h}(\cdot)$ given by \eqref{eq:gcexact} and of the iterates of $\varphi_{t}$ in \eqref{eq:iterPhit}, a straightforward singularity analysis provided in the companion Maple worksheet implies our claim.
\end{proof}

\section{Coalescence of geodesics in the UIPT} 
\label{sec:coalescence}
This section is devoted to the study of coalescence of discrete geodesics using the $\rho$-skeleton decomposition of the UIPT \cites{Kr,CLGfpp,M}. Our main result is Theorem \ref{thm:scale} below which generalizes and regroups the statements of \cite[Lemma 3.3]{CMM10} and \cite[Theorem 3.4]{CMM10} in the case of triangulations. As we will see, deducing Theorem \ref{thm:CMM} from it is an easy matter. Recall the definition of hulls of metric balls given in Section \ref{sec:hulls}.

\begin{definition}  \label{def:coalescence} In the UIPT we say that there is geodesic coalescence at scale $r$ if there exists a vertex $v_{r} \in \overline{B}_{2r} \backslash \overline{B}_{r}$ such that any geodesic $\gamma$ starting from $ \partial \overline{B}_{2r}$ and targeting the origin of the root edge must go through $v_{r}$.
\end{definition}

\begin{figure}[!ht]
 \begin{center}
 \includegraphics[width=6cm]{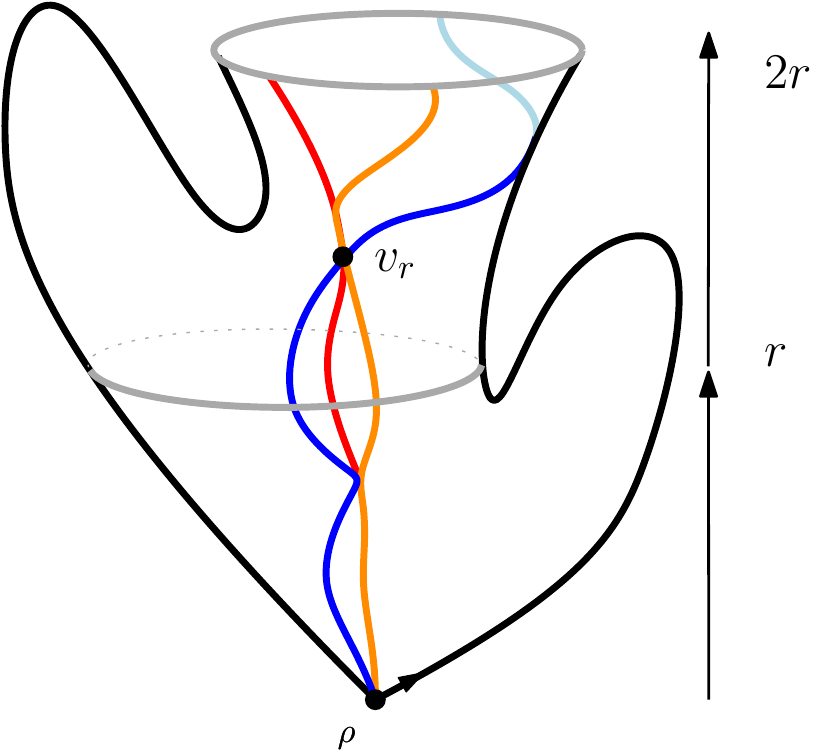}
 \caption{Illustration of the definition of geodesic coalescence at scale $r$. The three geodesics starting at level $2r$ and going down toward the origin must meet at the vertex $v_{r} \in \overline{B}_{2r}\backslash \overline{B}_{r}$. }
 \end{center}
 \end{figure}

\begin{theorem}[Geodesic coalescence]  \label{thm:scale}Almost surely there exists an infinite sequence $1<R_{1}<R_{2}~<~\dots$ such that there is geodesic coalescence at scale $R_{i}$ in $\mathbf{T}_{\infty}$ for all $i \geq 1$.\end{theorem}
\begin{remark}
As a trivial corollary of the preceding result, we deduce that in the UIPT there exists an infinite sequence of vertices $x_{1},x_{2}, ...$ such that for any $k \geq 1$, there exists $r>0$ such that for any point $y$ at distance at least $r$ from the origin $\rho$ of $\mathbf{T}_{\infty}$, any geodesic $\gamma$ between $\rho$ and $y$ must pass through $x_{1},x_{2}, ... , x_{k}$. This result is (a stronger) analog of the results \cite[Lemma 3.3]{CMM10} and \cite[Theorem 3.4]{CMM10} in the case of triangulations. 
\end{remark}
Let us derive Theorem \ref{thm:CMM} from Theorem \ref{thm:scale}:
\proof[Proof of Theorem \ref{thm:CMM}]  Clearly by invariance under re-rooting of the UIPT and following the proof in \cite{CMM10} it suffices to prove the theorem when $u$ is a neighbor of the origin (or even the target of the root edge). Using Theorem \ref{thm:scale} we can find two scales $R_{1}< 2 R_{1} < R_{2}$ such that there is geodesic coalescence at scales $R_{1}$ and $R_{2}$. Now, let $z$ be a point outside of $\overline{B}_{2R_{2}}$ and consider a geodesic $\gamma : z \to \rho$ and a geodesic $\gamma' : z \to u$. Clearly since $ |\mathrm{d_{gr}}(u,z) - \mathrm{d_{gr}}(\rho,z)| \leq 1$ then $\gamma'$ must always go down towards the origin of $\mathbf{T}_{\infty}$ (i.e.~the distance to the origin must strictly decrease) maybe except at one step (for otherwise going first to $\rho$ and then to $u$ would create a shorter path). Hence, $\gamma'$ coincides with a geodesic towards the origin either on $ \overline{B}_{2R_{2}} \backslash \overline{B}_{R_{2}}$ or on $ \overline{B}_{2R_{1}} \backslash \overline{B}_{R_{1}}$. Consequently $\gamma$ and $\gamma'$ meet either at the vertex $v_{R_{1}}$ or at the vertex $v_{R_{2}}$ with the notation of Definition \ref{def:coalescence}. We deduce in particular that 
$$ \mathrm{d_{gr}}(z, u) \geq \min \big( \mathrm{d_{gr}}(u,v_{R_{1}}) + \mathrm{d_{gr}}(v_{R_{1}},z) ; \mathrm{d_{gr}}(u,v_{R_{2}}) + \mathrm{d_{gr}}(v_{R_{2}},z)\big),$$
 and since the other inequality is obvious we deduce that 
 $$ \mathrm{d_{gr}}(z,u)-\mathrm{d_{gr}}(z,\rho) = \min \big( \mathrm{d_{gr}}(u,v_{R_{1}})- \mathrm{d_{gr}}(\rho,v_{R_{1}}) ; \mathrm{d_{gr}}(u,v_{R_{2}})- \mathrm{d_{gr}}(\rho,v_{R_{2}})\big),$$ for all $z$ outside of the hull of radius $2R_{2}$. This completes the proof of the result. \endproof

The proof of Theorem \ref{thm:scale} will use Krikun's $\rho$-skeleton decomposition. More precisely we identify a geometric event (Figure \ref{fig:geoevent}) happening within the $\rho$-skeleton which forces geodesic coalescence at scale $r$. But before doing that we need to describe the geodesics towards the origin within the $\rho$-skeleton.

\subsection{The $\rho$-skeleton of the UIPT}

In our proof of geodesic coalescence, we will use the description of the $\rho$-skeleton decomposition in the UIPT. More precisely we shall use the fact that the forest of trees describing the skeleton of hulls in $ \mathbf{T}_{\infty}$ is not too different from the law of iid $\theta$-Galton--{Watson} trees (this is very similar to \cite[Proposition 5]{CLGfpp}). We gather in this section the necessary background.\medskip

 Fix $r \geq 1$ and consider the skeleton decomposition of the triangulation of the cylinder $ \overline{B}_{2r} \backslash \overline{B}_{r}$ (we are thus measuring distances from the origin $\rho$). More precisely we will denote by $$ \mathcal{F}_{r} = (\tau_{1}, ... , \tau_{| \partial \overline{B}_{2r}|})$$ the forest of trees of maximum height $r$ describing it where we applied a uniform cyclic permutation (equivalently, the top root edge of $\overline{B}_{2r} \backslash \overline{B}_{r}$ is chosen uniformly at random) and where we forgot the distinguished point at height $r$ in the forest. For $q \geq 1$ let also $ \mathcal{I}(q) = (  \mathfrak{t}_{1}, ... , \mathfrak{t}_{q})$ be a forest of iid $\theta$-Galton--Watson trees cut at height $r$. Then we have 
 \begin{lemma} \label{lem:comparison} For any $ \varepsilon \in (0,1)$, there exists $C_{\varepsilon}>0$ such that for all $r\geq 1$ and $ p,q  \in [ \varepsilon r^{2};  \varepsilon^{-1} r^{2}]$ for any fixed forest $f_{0}$ of $q$  trees with $p$ vertices at maximum height $r$
 $$ \mathbb{P}( \mathcal{F}_{r} = f_{0} \mid | \partial \overline{B}_{2r}| = q) \geq C_{\varepsilon} \cdot \mathbb{P}( \mathcal{I}(q) = f_{0}).$$
 \end{lemma}
 \proof The law of $ \mathcal{F}_{r}$ conditionally on $| \partial \overline{B}_{2r}| = q$ is described in the third display in the proof of \cite[Proposition 5]{CLGfpp} and can be rewritten as follows: For any  fixed forest $f_{0}$ of $q$ trees with $p$ vertices at maximum height $r$ we have 
 \begin{eqnarray} \mathbb{P}( \mathcal{F}_{r} = f_{0} \mid | \partial \overline{B}_{2r}| = q) &=&  \frac{\mathbb{P}( |\partial \overline{B}_{r}|=p)}{ \mathbb{P}(| \partial \overline{B}_{2r}| = q)} \frac{h(q)}{h(p)} \prod_{v \in f_{0}^{\star}} \theta(c(v)), 
\end{eqnarray}
where $h(k)= 4^{-k}{\binom{2k}{k}}$. From \cite[Remark 3]{CLGfpp} and the fact that $ h(x) \sim 2/ \sqrt{\pi x}$ for large $x$ we deduce that when $p/r^{2}$ and $q/r^{2}$ are bounded away from $0$ and $\infty$ then the two ratios in the last display are bounded away from $0$ and infinity as well. Since the product in the last display exactly corresponds to $\mathbb{P}( \mathcal{I}(q) = f_{0})$ we are done. \endproof

\subsection{Left and right most geodesics}

\label{sec:geodesics}
We now describe the left and right most geodesics towards the origin of the UIPT using the $\rho$-skeleton decomposition. \medskip 
\
 
Take $v$ any vertex of the UIPT. It belongs to one of the triangulations with simple boundary, say $M$, filling a slot in the skeleton representation. Therefore, any geodesic path from this vertex to the origin goes first to one of the boundary vertices of $M$ and then proceeds to the origin by going through edges of the skeleton (\emph{i.e.} edges of the down triangles linking two consecutive layers), or eventually by going through edges linking boundary vertices of slots, see Figure \ref{fig:geo}.
In turn, this means that any geodesic path from $v$ to the origin stays on the right of the left-most geodesic from the top vertex of $M$ to the origin and stays on the left of the right-most geodesic from the top vertex of $M$ to the origin.
 
\begin{figure}[ht!]
\begin{center}
\includegraphics[width=1\textwidth]{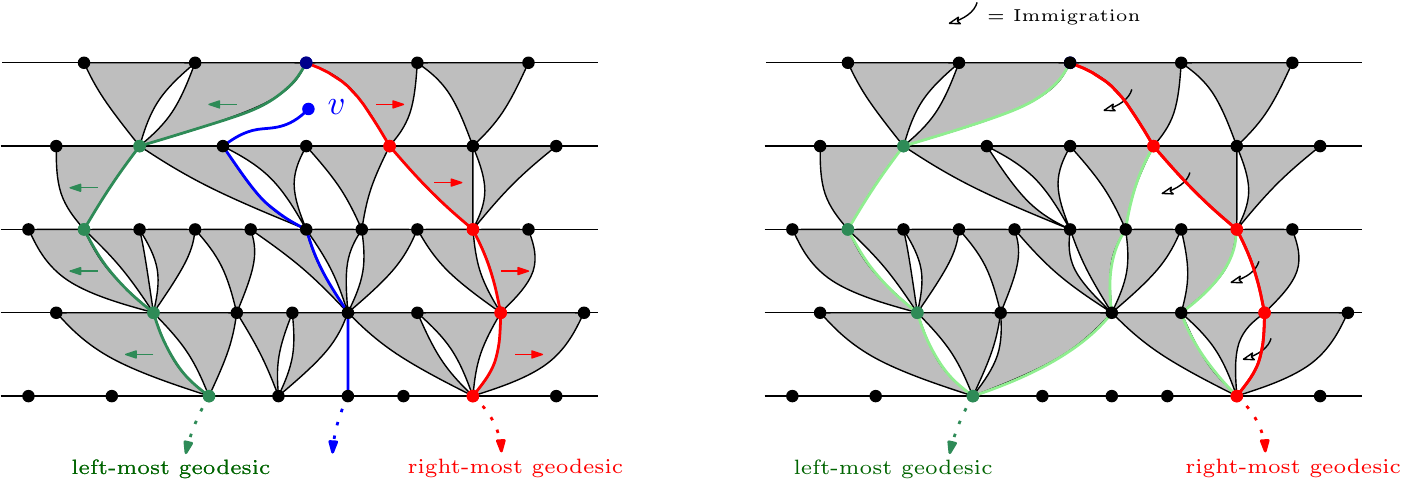}
\caption{\label{fig:geo} Left: A geodesic to the origin (blue), between a left-most geodesic (red) and a right-most geodesic (green). Notice the blue edge on the bottom of the geodesic from $v$ to the origin linking two boundary vertices of a slot. Right: The situation between a left-most geodesics (in green) and a right-most geodesic (in red): at each height the slot immediately on the right of the right-most geodesic allows for immigration in-between the geodesics. }
\end{center}
\end{figure}

Now imagine that we follow two left-most geodesics $\gamma_{1}$ and $\gamma_{2}$ targeting $\rho$ and started from height $r$ in the $\rho$-skeleton of the UIPT. These two geodesics trace the left and right boundaries of the two forests of trees in-between them (there are two forests since we are on a cylinder). In particular, if the minimum height of the two forests is $h <r$ this implies that $\gamma_{1}$ and $\gamma_{2}$ coalesce at height $r-h-1$. This phenomenon, already observed by Krikun \cite[Section 4.3]{Kr} was also instrumental in \cite[Proposition 16]{CLGfpp}: Since the trees in those forests are roughly $\theta$-Galton--Watson trees (see Lemma \ref{lem:comparison}), and in particular critical, they die out which ensures coalescence of geodesics.

Let us now imagine the situation when $\gamma_{1}$ is a left-most and $\gamma_{2}$ a right-most geodesic towards the origin. These two geodesics do not trace the boundaries of two separate forests in the skeleton anymore. More precisely, the trees in-between $\gamma_{1}$ and $\gamma_{2}$ now inherits, at each down step, the descendants of the vertex immediately on the right of $\gamma_{2}$, see Figure \ref{fig:geo} (right). Heuristically speaking  the horizontal length between $\gamma_{1}$ and $\gamma_{2}$ now evolves as a $\theta$-Galton--Watson tree together with an immigration at each generation distributed according to $\theta$. This description would be exact in an infinite model where the finite perimeters of the cylinders do not perturb the law (more precisely the Lower Half Plane Model of \cite{CLGfpp}).

\subsection{Enforcing geodesic coalescence}
\begin{figure}[ht!]
\begin{center}
\includegraphics[width=0.8\textwidth]{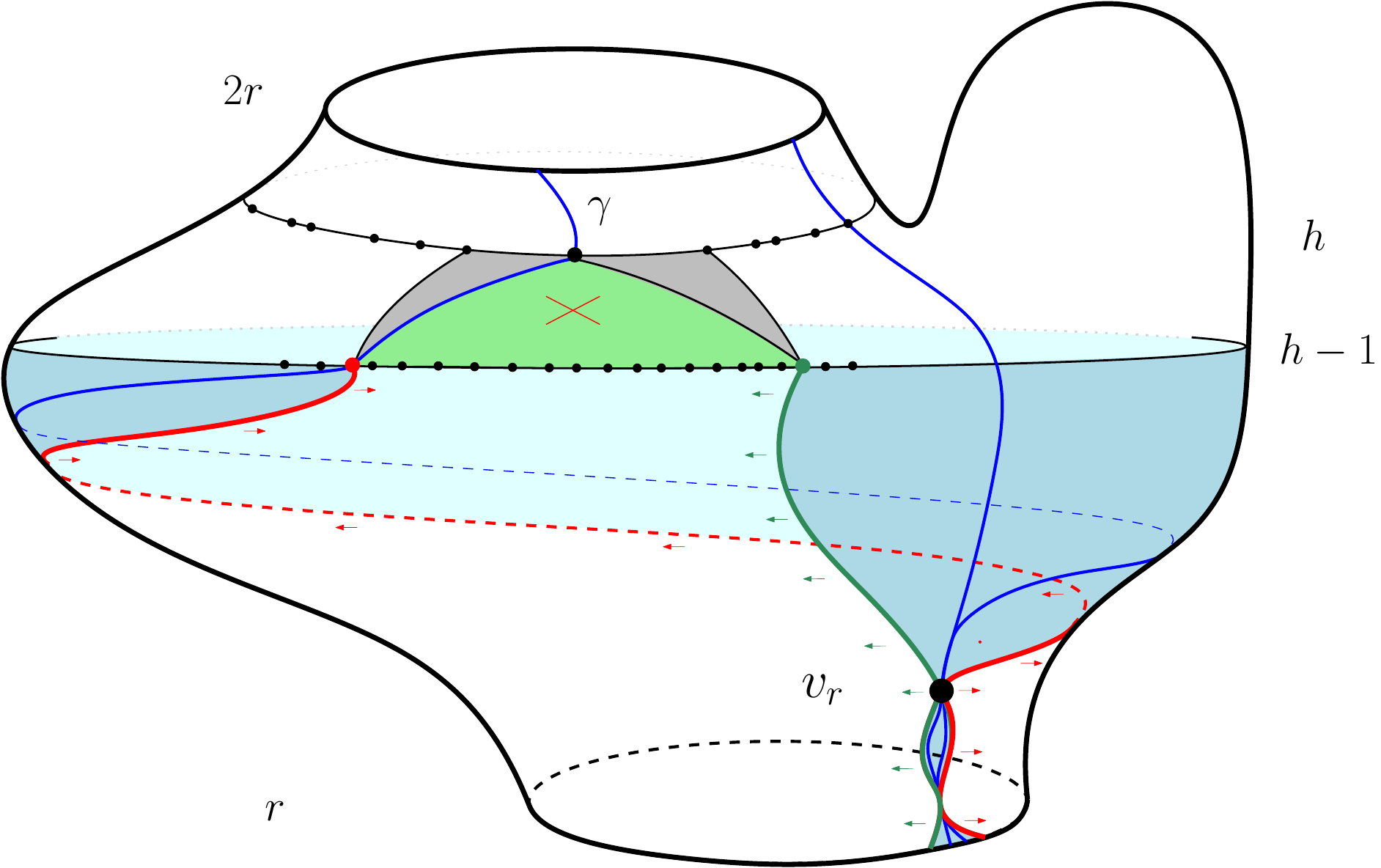}
\caption{\label{fig:geoevent} Illustration of the geometric event $ \mathcal{G}_{r}$ enforcing coalescence at scale $r$. We see the creation of a large slot (in green) whose perimeter is comparable to the total perimeter at level $h \geq 3r/2$. We then enforce that the right-most and left-most geodesic starting from the bottom corners of that slot meet before height $r$. Any geodesics towards the origin coming down from level $2r$ must live in the blue region and in particular must go through $v_{r}$.}
\end{center}
\end{figure}

We now describe the geometric event $ \mathcal{G}_{r}$ which, when happening in the $\rho$-skeleton between height $r$ and $2r$, enforces geodesic coalescence at scale $r$. It is defined by three conditions (see Figure \ref{fig:geoevent}):
 \begin{itemize} 
\item  $C1(r)$: There exists $h \in [3/2 r, 2r]$ such that there is a slot at height $h$ of perimeter $ s \in [r^{2}, 2r^{2}]$ 
and such that $| \partial \overline{B}_{h-1}| \in [3r^{2}, 4r^{2}]$. 
\item $C2(r)$: The critical Boltzmann triangulation filling-in the above slot has no edge between its origin vertex and its boundary vertices except the two neighbors of the origin vertex.
\item $C3(r)$: The right-most (red) geodesic starting from the bottom left of the above slot and the left-most (dark green) geodesic starting from the bottom right of the above slot meet at some vertex $v_{r}$ whose height is larger than $r$.
\end{itemize}

It is easy to see that conditions $C1(r)$ and  $C3(r)$ are conditions about the skeleton $ \mathcal{F}_{r}$ of $ \overline{B}_{2r} \backslash \overline{B}_{r}$ and so by abuse of notion we will also say that a fixed forest $f _{0}$ of height $r$ satisfies  $ C1(r) \cap C3(r)$ if it is the forest describing the skeleton of $\overline{B}_{2r}( \mathfrak{t}) \backslash \overline{B}_{r}( \mathfrak{t})$ where $ \mathfrak{t}$ is a triangulation satisfying $C1(r)$ and $C3(r)$.  The only condition about the slots (and thus about the UIPT) is condition $C2(r)$. However, recall that in the UIPT,  conditionally on the skeleton the slots are filled-in with independent Boltzmann triangulations with the proper perimeter: the following lemma shows that $C2(r)$ is fulfilled with a probability bounded away from $0$ irrespectively of the geometry of the skeleton.
\begin{lemma}\label{lem:boltz} The probability that a critical Boltzmann triangulation of the $p-$gon with $p \geq 2$ has no edge between its origin vertex and its boundary vertices except the two neighbors of the origin is larger than $ \frac{1}{12 \sqrt{3}}$.
\end{lemma}
\begin{proof}
Consider a triangulation of the $p$-gon $ \mathfrak{t}$ obtained as the gluing of a triangulation of the $p$-gon $ \mathfrak{t}'$ together with two triangles as depicted in the Figure \ref{fig:needle}. Then clearly such a triangulation $ \mathbf{t}$ has no edge between the origin vertex and the other vertices of the boundary except those neighboring the origin. The total weight under the Boltzmann distribution of such triangulations is $ x_c T_{p}( x_c ) / T_{p}(x_c) = x_c = \frac{1}{12 \sqrt{3}}$.
\begin{figure}[!h]
 \begin{center}
 \includegraphics[width=0.4\linewidth]{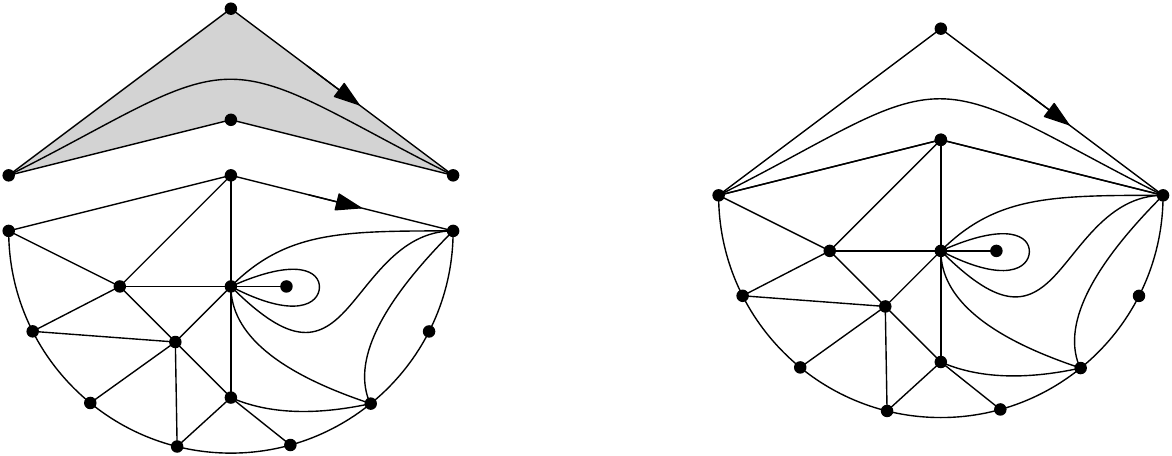}
 \caption{ \label{fig:needle}Illustration of the construction of a triangulation of the $p$-gon having no edge between the origin and non adjacent boundary vertices.}
 \end{center}
 \end{figure}
\end{proof}

\begin{lemma} On the event $  \mathcal{G}_{r}$ there is geodesic coalescence at scale $r$.
\end{lemma}
\proof This should be clear on Figure \ref{fig:geoevent}. Let $\gamma$ be a geodesic coming down from level $2r$ and targeting the origin of the map. We claim that from level $h-1$ down to level $r$ the geodesic $\gamma$ must be located in-between the two extremal left-most and right-most geodesics starting from the bottom of the slot selected in $ \mathcal{G}_{r}$. This is true at level $h-1$ by condition $C2(r)$ since the geodesic $\gamma$ cannot ``cross'' the green slot and then propagates by the definition of left-most and right-most geodesic (see last section). Since by $C3(r)$ these two extreme geodesics meet at some point $v_{r}$, the same holds for $\gamma$. \endproof 

\begin{proposition} \label{prop:positive} There exists $c>0$ such that for any large enough $r$ we have $ \mathbb{P}( \mathcal{G}_{r}) > c$. More precisely, for any $ \varepsilon>0$ we can find $ \delta >0$ such that for all $r$ large enough 
$$  \mathbb{P} \bigg( \mathbb{P} \big( \mathcal{G}_{r} \mid \big| \partial \overline{B}_{2r} \big| \big) > \delta \bigg) \geq 1- \varepsilon.$$
\end{proposition}
\proof Clearly by Lemma \ref{lem:boltz} it is sufficient to take care of conditions $C1(r)$ and $C3(r)$ which are conditions on the forest $ \mathcal{F}_{r}$ describing  the skeleton of $ \overline{B}_{2r} \backslash \overline{B}_{r}$. Fix $ \varepsilon>0$. Using \cite[Remark 3]{CLGfpp} we can find $A>2$ large enough so that $ \mathbb{P}(|\partial \overline{B}_{2r}| \in [A^{-1}r^{2},Ar^{2}]) \geq 1 - \varepsilon$. We will show that we can find $\delta >0$ such that if $q \in  [A^{-1}r^{2},Ar^{2}]$ then  $\mathbb{P} \big( C1(r) \cap C3(r) ~\big | ~ | \partial \overline{B}_{2r} | = q \big) > \delta$.  
Actually, we will even show that if $C4(r)$ is the event on which $| \partial \overline{B}_{r}| \in [A^{-1}r^{2}, Ar^{2}]$ then 
$$\mathbb{P} \big( C1(r) \cap C3(r)  \cap C4(r) \mid | \partial \overline{B}_{2r} \big| = q \big) > \delta.$$
We added condition $C4(r)$ because thanks to  Lemma \ref{lem:comparison}, in order to prove the above display it suffices to show the similar estimate when we replace the skeleton of $ \overline{B}_{2r} \backslash \overline{B}_{r}$ by a forest of $q$ independent $\theta$-Galton--Watson trees. In order to lighten the prose we will implicitly assume that we have performed this change, i.e.~we suppose in the remaining calculations that the forest coding of $\overline{B}_{2r} \backslash \overline{B}_{r}$ has the same law as $ \mathcal{I}(q)$. By abuse of notation we still use the notation $B_r,|\partial B_r|$ ... but the reader should keep in mind that everything we compute only depends on the underlying forest.

Let us first see condition $C1(r)$: The process $(| \partial \overline{B}_{2r-n}|)_{n \geq 0}$ is now a $\theta$-Galton--Watson process started from $q$ particles that we let evolve during $ \approx r$ steps. Since $ q \in [A^{-1}r^2,Ar^2]$ we are exactly within the scaling window to get convergence towards the stable-$3/2$ continuous state branching process (see \cite{DLG02}). The fact that $C1(r)$ happens with a uniform positive probability follows from the last convergence  together with the fact that the topological support of the trajectories of  the $3/2$-stable continuous state branching process is the set of all c\`adl\`ag functions tending to $0$ at infinity.

We now turn to condition $C3(r)$: Suppose that $C1(r)$ happens and pick the largest height  $h \geq 3r/2$ at which  a slot of perimeter $s \in [r^2,2r^2]$ occurs and such that $|\overline{B}_{h-1}| \in  [3r^{2}, 4r^{2}]$. We now start the left and right most geodesics $\gamma_{ \ell}$ and $\gamma_{r}$ from the two extreme lower points of the selected slot. By the Markov property (we are going downward in height) and from the discussion in Section \ref{sec:geodesics} we claim that the horizontal width between the left-most and the right-most geodesic (as we go down) evolves as a $\theta$-Galton--Watson process with immigration $\theta$. More precisely, we introduce the process $$(X_{n}, Y_{n}, Z_{n})_{0 \leq n \leq h-1-r}$$ where $X_{n}$ is the total number of edges in $| \partial \overline{B}_{h-1-n}|$, whereas $Y_{n}$ is the number of edges in-between $\gamma_{r}$ and $\gamma_{\ell}$ in cyclic order and $Z_{n}$ those in-between $\gamma_{\ell}$ and $\gamma_{r}$ (the blue region in Figure \ref{fig:geoevent}) at level $h-1-n$.  Since we supposed that the underlying forest has law $ \mathcal{I}(q)$ this is a Markov process whose evolution is prescribed as follows 
$$ Y_{0} = s-2 \quad \mbox{ and } \quad  Z_{0} = | \partial \overline{B}_{h-1}|-(s-2),$$
$$
X_{n}=Y_{n}+Z_{n} \qquad \mbox{ for } 0 \leq n \leq h-1-r,$$
and for $0 \leq n \leq h-2-r$, as long as $Y_{n}>0$ we have 

$$ Y_{n+1} = \sum_{i=1}^{Y_{n}-1} \xi^{(n)}_{i} \quad \mbox{ and } \quad 
Z_{n+1} = \sum_{i=1}^{Z_{n}+1} \tilde{\xi}^{(n)}_{i},$$
where $\tilde{\xi}^{(n)}_{i}, \xi^{{(n)}}_{i}$ are i.i.d.~random variables of law $\theta$. Notice that each step, one particle (and thus its whole descendance after that step) emigrates from the population $Y$ to the the population $Z$. The next lemma shows that condition $C3(r) \cap C4(r)$ has a positive probability to happen conditionally on $C1(r)$ finishing the proof of the proposition.  
\begin{lemma} \label{lem:immigration} There exists a constant $c \in (0,1)$ such that for any $r \geq 1$, any $y_{0},z_{0} \in [r^{2}/4, 4r^{2}]$ and $n_{0} \in [r/4, 4r]$, if we consider the above process $(X_{n},Y_{n},Z_{n})$ started from $Y_{0}=y_{0}$ and $Z_{0}=z_{0}$, then  we have
\[
\mathbb P\left( Y_{i}>0, \forall 0 \leq i \leq r_{0} \mbox{ and } Z_{n_{0}} = 0 \mbox{ and } Y_{n_{0}}  \in [r^{2}, 2 r^{2}] \right) \geq c.
\]
\end{lemma}
\begin{proof} Clearly as long as $Y>0$, the sum $X$ evolves as a $\theta$-branching process. As already mentioned, the processes $Y$ and $Z$ both evolve as $\theta$-branching processes except that at each time step, a particle of the population $Y$ emigrates towards population $Z$.

We consider first the families descending from the emigrated particles from $Y$ at time $0,1,2,...$, see Figure \ref{fig:immigrationbis}. If $\zeta^{(i)}_{n}$ denotes the number of descendants at time $n \geq i$ of the particle emigrated from $Y$ at time $i$ then $(\zeta^{(i)}_{i+k})_{k \geq 0}$ are independent $\theta$-branching processes\footnote{to be precise, if $Y$ dies out we imagine that we create a new particle to run $\zeta^{(i)}$}. We claim that these particles exchanged between $Y$ and $Z$ do not affect too much the process at time $n_{0}$, more precisely if we put $$ \mathcal{E}_{n} := \sum_{i=0}^{n-1} \zeta^{(i)}_{n},$$ then there exists $c>0$ such that
 \begin{eqnarray} \label{eq:claims} (i) \ \  r^{-2} \sup_{0 \leq i \leq n_{0}} \mathcal{E}_{i}  \xrightarrow[r\to\infty]{} 0 \quad \mbox{and } \quad (ii)\ \ \mathbb{P}(\mathcal{E}_{n_{0}}=0) > c.  \end{eqnarray}
For the second claim we compute explicitly : $ \mathbb{P}(\zeta^{(i)}_{i+r}=0) = 
\varphi_1^{\{r\}} (0) = 1 - (r+1)^{-2}$ and so 
$$ \mathbb{P}(\mathcal{E}_{n_{0}}=0) = \prod_{r=1}^{n_{0}} \left( 1 - \frac{1}{(r+1)^{2}}\right) \geq \prod_{r=1}^{\infty} \left( 1 - \frac{1}{(r+1)^{2}}\right) >0.$$
For the first claim, notice that since $\theta$ is critical then $(\mathcal{E}_{i}-i)_{i \geq 0}$ is a martingale. We can then apply Doob maximal inequality and get that 
$$ \mathbb{P}( \sup_{0 \leq i \leq n_{0}} \mathcal{E}_{i} >  \varepsilon r^{2}) \leq \mathbb{P}( \sup_{0 \leq i \leq n_{0}} (\mathcal{E}_{i}-i) >  \varepsilon r^{2} - n_{0}) \leq  \frac{ \mathbb{E}[|\mathcal{E}_{n_{0}}-n_{0}|]}{ \varepsilon r^{2} - n_{0}} \leq \frac{2n_{0}}{ \varepsilon r^{2}-n_{0}},$$ and the claim follows since the last ratio goes to $0$ as $r \to \infty$.

\begin{figure}[!h]
 \begin{center}
 \includegraphics[width=0.8\linewidth]{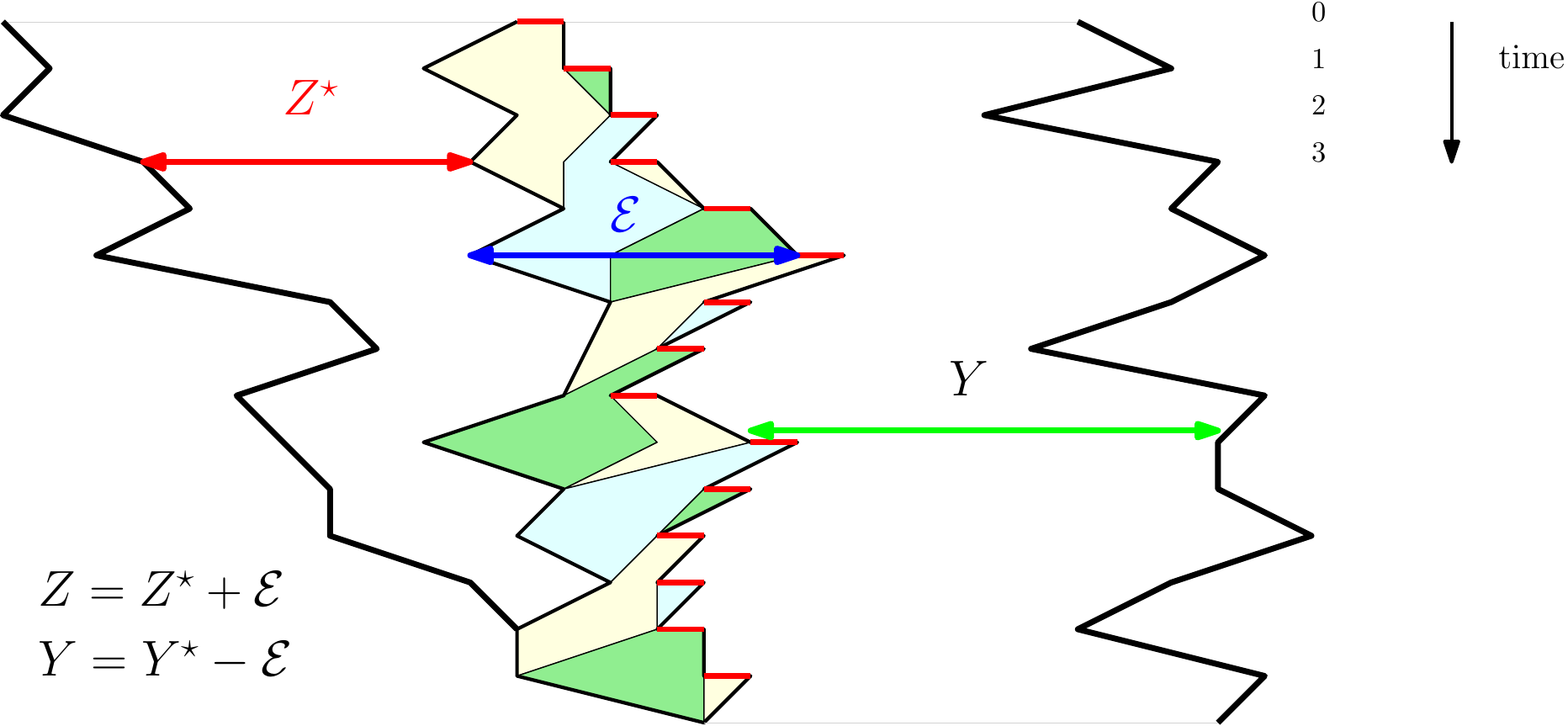}
 \caption{\label{fig:immigrationbis}Illustration of the proof of Lemma \ref{lem:immigration} : at each generation step, one particle (in red) is stolen to $Y$ to be reassigned to $Z$. The families of offspring of these migrants are depicted with different colors. Their aggregation forms the process $\ \mathcal{E}$ which is negligible compared to $Y$ and $Z$.  }
 \end{center}
 \end{figure}

We can now prove the lemma. Imagine first that we start a standard $\theta$-branching process ${Y}^{\star}$ started from $ r^{2}/4 \leq y_{0} \leq 4 r^{2}$ particles and run it for time $ r/4 \leq n_{0} \leq 4r$. The convergence of this branching process towards a $3/2-$stable continuous state branching process ensures that there is a positive probability that the process $Y^\star$ ends up with $ k \in [r^{2}/4,4r^{2}]$ particles at time $n_0$ and furthermore never drops below $r^{2}/8$ before this time. On such an event, imagine that one particle (and its whole descendence) emigrates at each generation to form the process $(\mathcal{E}_{n} : n \geq 0)$ so that in our notation we have $Y = Y^{\star}-\mathcal{E}$.  By \eqref{eq:claims} this is possible i.e.~even after each emigration we still have $Y>0$ and we can furthermore suppose that $ \mathcal{E}_{n_{0}}=0$ with positive probability.

We now start independently a $\theta$-branching process $(Z_{n}^{\star})$ from $z_{0}$. The probability that this process dies out before time $n_{0}$ is $ (1- (n_{0}+1)^{-2})^{z_{0}} > c >0$. On the combination of the above events whose probability is bounded away from $0$ the processes $Y = Y^{\star}-\mathcal{E}$ and $Z = Z^{\star}+ \mathcal{E}$ have the desired law and the event of the lemma is realized. This proves our claim.
\end{proof}

\subsection{Proof of Theorem \ref{thm:scale}}
The expert reader may already be convinced that since $ \mathbb{P}( \mathcal{G}_{r})$ is uniformly bounded from below,  a kind of ``asymptotic independence of the scales in the UIPT'' must show that  $ \mathcal{G}_{i}$ happens infinitely often with probability $1$, thus implying Theorem \ref{thm:scale}. We present here a complete proof based on Jeulin's lemma \cite{Jeu82}. 

Consider the random variables $ \mathcal{X}_{r}= \mathbb{P}( \mathcal{G}_{r} \mid | \partial \overline{B}_{2r}|)$ for $r \geq 1$. Those variables are not independent but we know from Proposition \ref{prop:positive} that for any $ \varepsilon \in (0,1)$ we can find $\delta>0$ such that for all $r \geq 1$ we have
$$ \mathbb{P}( \mathcal{X}_{r} \geq \delta) \geq 1- \varepsilon.$$ We claim that this is sufficient to apply Jeulin's lemma \cite[Proposition 4 c]{Jeu82} and deduce that 
 \begin{eqnarray} \sum_{i=1}^{\infty} \mathcal{X}_{2^{i}} = \sum_{i =1}^{\infty} \mathbb{P}( \mathcal{G}_{2^{i}} \mid | \partial \overline{B}_{2^{i+1}}|) = \infty, \quad \mbox{ almost surely}.   \label{eq:infinity}\end{eqnarray}
 Since the proof is short, let us include it here for the sake of completeness: Suppose by contradiction that the above series is finite, say less than $M>0$, on an event $ \mathcal{A}$ of probability at least $\varepsilon>0$. Using Proposition \ref{prop:positive} we find $\delta >0$ so that $\mathbb{P}( \mathcal{X}_{r} \geq \delta) \geq 1- \varepsilon/2$ for every interger $r$. Noticing that $ \mathbb{P}( \mathcal{A} \cap \{ \mathcal{X}_{r} \geq \delta\}) \geq \varepsilon/2$ we can write 
 $$ M \geq \mathbb{E}\left[ \sum_{i=1}^{\infty} \mathcal{X}_{2^{i}} \mathbf{1}_{ \mathcal{A}}\right] = \sum_{i=1}^{\infty}  \mathbb{E}[ \mathcal{X}_{2^{i}} \mathbf{1}_{ \mathcal{A}} ] \geq  \sum_{i=1}^{\infty}  \delta \mathbb{E}[  \mathbf{1}_{ \mathcal{X}_{2^{i}} \geq \delta} \mathbf{1}_{ \mathcal{A}} ] \geq \sum_{i=1}^{\infty} \delta \frac{ \varepsilon}{2},$$ which is clearly a contradiction. Hence \eqref{eq:infinity} holds.\medskip

Now, imagine that we explore the $\rho$-skeleton of the UIPT starting from $\infty$ and going down to $0$. From the results of Krikun \cites{Kr,CLGfpp} this exploration evolves as an (inhomogeneous) Markov process. When we have discovered the skeleton from $\infty$ up to height $2^{i+1}$, by the Markov property, the probability that $ \mathcal{G}_{2^{i}}$ is satisfied is given by $ \mathbb{P}( \mathcal{G}_{2^{i}} \mid | \partial \overline{B}_{2^{i+1}}|)$. Notice now, still by the Markov property of the exploration, that conditionally on $ | \partial \overline{B}_{2^{i}}|$ the event $  \mathcal{G}_{2^{i-1}}$ is independent of the skeleton above height $2^{i}$ and in particular of $ \mathcal{G}_{2^{i}}$. Iterating the argument we deduce that 
$$ \sum_{i =1}^{\infty}  \mathbf{1}_{\mathcal{G}_{2^{i}}}  \overset{(d)}{=}  \sum_{i=1}^{\infty} \mathrm{Bernoulli}(\mathbb{P}( \mathcal{G}_{2^{i}} \mid | \partial \overline{B}_{2^{i+1}}|)),$$ in distribution where all the Bernoulli random variables are independent. According to the Borel--Cantelli lemma and \eqref{eq:infinity} the last series is infinite almost surely. This proves Theorem \ref{thm:scale}. \qed

\section{The skeleton seen from infinity} \label{sec:skelinfinity}
We now use our results on coalescence of geodesic to construct the skeleton seen from infinity in the UIPT as the limit of the $\delta$-skeletons of  large uniform pointed triangulations and compute the distribution of the perimeter and volume of horohulls in the scaling limit.

\subsection{The $\delta$-skeleton of the UIPT}

Let $( \mathbf{T}_{n}, \delta_{n})$ be a uniform triangulation of the $1-$gon with $n$ inner vertices, pointed at a uniform inner vertex. We denote its origin by $\rho_{n}$ and  for $u \in \mathrm{Vertices}( \mathbf{T}_{n})$ define $\ell_{n}(u) = \mathrm{d_{gr}}(u, \delta_{n})- \mathrm{d_{gr}}(\rho, \delta_{n})$. We perform the $\delta_n$-skeleton decomposition in $(\mathbf{T}_{n},\delta_n)$ as in Section \ref{sec:skelretourne} and denote by  $ \mathsf{Skel}_{n} \subset \mathbf{T}_{n}$ the tree coding this skeleton. We also denote by $(M^{(n)}_{v} : v \in \mathsf{Skel}_{n})$ the triangulations filling-in the slots. 

\begin{theorem} \label{thm:skeletoninfinity} We have the following convergence in distribution for the local topology\footnote{this simply means that for any $r\geq 0$ if we restrict on the ball of radius $r$ around the origin of the map then the corresponding restrictions converge in distribution}
$$ \left( \mathbf{T}_{n}, \ell_{n}, \mathsf{Skel}_{n} , \big( M_{v}^{(n)} : v \in \mathsf{Skel}_{n} \big)\right) \xrightarrow[n\to\infty]{(d)} \left( \mathbf{T}_{\infty}, \ell, \mathsf{Skel}, \big( M_{v} : v \in \mathsf{Skel} \big)\right),$$
where $\ell$ is the horofunction defined in Theorem \ref{thm:CMM} and $(\mathsf{Skel}, ( M_{v} : v \in \mathsf{Skel} )$ is the ``skeleton decomposition of the UIPT seen from infinity''. More precisely, conditionally on $ \mathsf{Skel}$
\begin{itemize}
\item The maps $(M_{v} : v \in \mathsf{Skel})$ are independent,
\item If $v \ne \rho$ then $M_{v}$ is a critical Boltzmann triangulation of the $(c(v)+2)$-gon and $M_{\rho}$ is a critical Boltzmann $\delta$-cap with perimeter $c(\rho)$.
\item The distribution of the random plane tree $ \mathsf{Skel}$ is characterized as follows. For any plane tree $\tau_{0}$ of height $r$ with $p \geq 1$ vertices at generation $r$ we have
$$ \mathbb{P}([ \mathsf{Skel}]_{r} = \tau_{0}) = p \cdot \frac{\nu( c(\rho))}{2} \cdot \prod_{v \in \tau_{0}^{\star}\backslash \rho} \theta( c(v)),$$
where $[ \mathsf{Skel}]_{r}$ is the tree $ \mathsf{Skel}$ truncated at generation $r$.
\end{itemize}
\end{theorem}
\begin{remark}
Since the expectation of $\nu$ is $2$ and $\theta$ is critical, the law of $ \mathsf{Skel}$ as described above obviously coincides with the definition given in the introduction in terms of conditioned (multi-type) Galton--Watson tree, see \cite{ADG}.\end{remark}
\proof  $ \boldsymbol{\mathsf{T}_{n} \to \mathsf{T}_{\infty}}$. The convergence $ \mathbf{T}_{n} \to \mathbf{T}_{\infty}$ for the local topology is the well-known result of Angel \& Schramm (to be precise, see \cite{Ste14} or \cite{CLGfpp} for the case of type I triangulations). Since $ \mathbf{T}_{\infty}$ is locally finite this implies that $ \mathrm{d_{gr}}(\rho_{n}, \delta_{n}) \to \infty$ in probability. By appealing to the Skorokhod representation theorem we will assume in the rest of this proof that $ \mathbf{T}_{n} \to \mathbf{T}_{\infty}$ and $ \mathrm{d_{gr}}( \rho_{n}, \delta_{n}) \to \infty$ almost surely. We will now show that the other convergences hold almost surely.

$\boldsymbol{\ell_{n} \to \ell}$. Fix $r_{0} \geq 1$, by Theorem \ref{thm:CMM} we know that there exists $R_{0} \geq 1$ so that for any $z \notin B_{R_{0}}$ and for every $u \in B_{r_{0}}$ we have 
$$ \ell(u) = \mathrm{d_{gr}}(u,z)- \mathrm{d_{gr}}(\rho,z).$$
A moment's thought shows that as soon as $ \mathbf{T}_{n}$ and $ \mathbf{T}_{\infty}$ coincide over the balls of radius $R_{0}+1$ around their origins then the last display holds in $ \mathbf{T}_{n}$ with $\ell$ replaced by $\ell_{n}$.  Since $ \mathrm{d_{gr}}(\rho_{n}, \delta_{n}) \to \infty$ we can eventually take $z= \delta_{n}$  and deduce that $\ell$ eventually coincide with $\ell_{n}$ on the ball of radius  $r_{0}$. This proves the convergence of the second coordinate.

$\boldsymbol{ \mathsf{Skel}_{n} \to \mathsf{Skel}}$. Here again, by the construction of Section \ref{sec:skelretourne}, it is easy to see that the local convergence of the skeleton (and of the slots) is a deterministic consequence of the convergence $( \mathbf{T}_{n}, \ell_{n}) \to ( \mathbf{T}_{\infty},\ell)$ provided that we know the horohulls of radius $r$ in $ \mathbf{T}_{\infty}$ are finite almost surely. To prove this, we shall directly compute the law of the first $r$ generations of the skeleton decomposition in $ \mathbf{T}_{n}$ and observe that it converges towards the law described in the proposition.

\textbf{Computation of the law.} Fix $\Delta$ a triangulation of the cylinder of height $r \geq 0$ and $\Delta_1$ a compatible $\delta$-cap so that gluing $\Delta_1$ to the top cycle of $\Delta$ gives a horohull of perimeter $p\geq 1$ that will be denoted by $\Delta_1 \cup \Delta$. Recall that $T_1$ is the generating series of triangulations of the $1$-gon counted by inner vertices and that for a planar map $m$ the quantity $|m|$ denotes its \emph{total} number of vertices. We have: 
\begin{align*}
\mathbb P & \left( \overline{H}_{r+1} (\mathbf T_n ,\delta_n) = \Delta_1 \cup \Delta \right) \\
& \quad = 
\frac{\left[ x^{n+1-(|\Delta_1 \cup \Delta|-p)}\right] \sum_{k\geq 1} C_{k,0,p} (x) }{n \left[ x^{n} \right] T_1 (x)}\\
& \quad = \frac{
x_c^{-n+|\Delta_1 \cup \Delta|-1-p} 
\left[ s^{n+1-(|\Delta_1 \cup \Delta|-p)}\right] (x_c \cdot s) (y_c \cdot t)^p
\sum_{k\geq 1} \left( \left(\varphi_t^{\{k\}} (0) \right)^p - \left(\varphi_t^{\{k-1\}} (0) \right)^p \right) }
{n \left[ x^{n} \right] T_1 (x)}\\
& \quad =
\frac{
x_c^{-n+|\Delta_1 \cup \Delta|-p} 
\left[ s^{n-|\Delta_1 \cup \Delta|+p}\right] (y_c \cdot t)^p
}
{n \left[ x^{n} \right] T_1 (x)}
\end{align*}
Since $t = 1 - \sqrt{\frac{2}{3}} (1-s)^{1/2} + \mathcal O (1-s)$ as $s\to 1$, we have, as $n \to \infty$, by standard singularity analysis:
\begin{align}
\mathbb P \left( \overline{H}_{r+1} (\mathbf T_n ,\delta_n) = \Delta_1 \cup \Delta \right)
&\sim
\frac{
x_c^{-n + |\Delta_1 \cup \Delta| - p } 
\frac{1}{\sqrt{6 \pi}} \,p \, y_c^p \left(n-|\Delta_1 \cup \Delta|+p \right)^{-3/2}
}
{x_c^{-n} \frac{1}{6 \sqrt{2 \pi}} n^{-3/2}}\nonumber \\
&\rightarrow 2 \sqrt{3} \, p \, y_c^p \, x_c^{|\Delta_1 \cup \Delta|-p} \label{eq:interessant}.
\end{align}

From there, we just have to transform the above expression into the announced law using the $\delta$-skeleton decomposition of $\Delta_1 \cup \Delta$. Denoting by $\tau$ the tree coding this skeleton we have
\begin{align*}
\mathbb P &\left( \overline{H}_{r+1} (\mathbf T_\infty) = \Delta_1 \cup \Delta \right)\\
& \quad =
p \, \left(2\sqrt{3} \left(\frac{y_c}{x_c}\right)^{c(\rho)} K_{c(\rho)} (x_c)\right) \, \left(\prod_{v \in \tau^\star\setminus \rho} \theta(c(v))\right) \, 
\left( \frac{x_c^{|\Delta_1|}}{K_{c(\rho)} (x_c)} \right)
\,
\left(\prod_{v \in \tau^\star\setminus \rho} \frac{x_c^{ \mathrm{inn}(M_v)}}{T_{c(v)+2}(x_c)} \right)
\end{align*}
which is exactly what we want: the first parenthesis is $\nu(c(\rho))/2$ (recall equation \eqref{eq:nuval}) and the others are easy to recognize.
\endproof
\begin{remark} It is interesting to notice in \eqref{eq:interessant} that $ \mathbb{P}( \overline{H}_{r} = h_{0})$ where $h_{0}$ is a given horohull of radius $r$ takes the form $x_{c}^{|h_{0}|} f(|\partial h_{0}|)$ for some function $f$. This is reminiscent of the spatial Markov property of hull of balls in the UIPT, see the definition of Markovian triangulations of the plane in \cite[Eq (1)]{CurPSHIT}.\end{remark}

\subsection{Scaling limits for horohulls in the UIPT}
\label{sec:scaling}

This section is devoted to the proof of Proposition \ref{thm:scale} which is mostly a computation.
From the previous section, if $F \in \mathcal{F}(r,q,p)$ is a forest of height $r$ with given perimeters, we have for any $s_1,s_2 \in [0,1]$ and denoting by $t_1$ and $t_2$ their respective conjugates defined by relation \eqref{eq:paramst}:
\[
\mathbb E \Big[ 
s_1^{|\overline{H}_{r+1} (\mathbf \mathbf{T}_{\infty})|} s_2^{|\partial \overline{H}_{r+1} (\mathbf \mathbf{T}_{\infty})|}
\Big| \mathrm{Skel} \left( \overline{H}_{r+1} (\mathbf \mathbf{T}_{\infty}) \setminus \overline{H}_{1} (\mathbf \mathbf{T}_{\infty})  \right) = F
\Big]
= s_2^q \frac{K_p(x_c s_1)}{K_p (x_c)} s_1^{q-p} \prod_{v \in F^\star} \frac{s_1 T_{c(v) + 2}(x_c s_1)}{T_{c(v) + 2}(x_c)}.
\]
Since $\prod_{v \in F^\star} t_1^{c(v)} = t_1^{q-p}$, we have
\[
\mathbb E \Big[ 
s_1^{|\overline{H}_{r+1} (\mathbf \mathbf{T}_{\infty})|} s_2^{|\partial \overline{H}_{r+1} (\mathbf \mathbf{T}_{\infty})|}
\Big| \mathrm{Skel} \left( \overline{H}_{r+1} (\mathbf \mathbf{T}_{\infty}) \setminus \overline{H}_{1} (\mathbf \mathbf{T}_{\infty})  \right) = F
\Big]
= s_2^q \frac{K_p(x_c s_1)}{K_p (x_c)} \left( \frac{s_1}{t_1}\right)^{q-p}\prod_{v \in F^\star} \frac{[z^{c(v)}] \varphi_{t_1}}{[z^{c(v)}] \varphi}.
\]
Summing over every forest of $\mathcal{F}(r,q,p)$ gives
\begin{align*}
\mathbb E \Big[ 
s_1^{|\overline{H}_{r+1} (\mathbf \mathbf{T}_{\infty})|} s_2^{|\partial \overline{H}_{r+1} (\mathbf \mathbf{T}_{\infty})|}
& \, \mathbf{1}_ 
{ \{ |\partial \overline{H}_{r+1} (\mathbf \mathbf{T}_{\infty})| = q , |\partial \overline{H}_{1} (\mathbf \mathbf{T}_{\infty})| = p \} }
\Big]\\
& = 2 \sqrt{3} \left( \frac{y_c t_1}{x_c s_1}\right)^{p}
K_p(x_c s_1) \, 
q \left( \frac{s_1 s_2}{t_1}\right)^{q}[z^q] \left( \varphi_{t_1}^{\{r\}} \right)^p.
\end{align*}
Summing over $p$ and recalling the expression \eqref{eq:Ksz} of $K$ gives:
\begin{align*}
\mathbb E \Big[ 
s_1^{|\overline{H}_{r+1} (\mathbf \mathbf{T}_{\infty})|} s_2^{|\partial \overline{H}_{r+1} (\mathbf \mathbf{T}_{\infty})|}
& \, \mathbf{1}_ 
{ \{ |\partial \overline{H}_{r+1} (\mathbf \mathbf{T}_{\infty})| = q \} }
\Big]\\
& = 2 \sqrt{3} \,
q \left( \frac{s_1 s_2}{t_1}\right)^{q}[z^q] K \big( s_1,\varphi_{t_1}^{\{r\}}  \big)\\
& = 2 \sqrt{3} \,
q \left( \frac{s_1 s_2}{t_1}\right)^{q}[z^q]
\Bigg(
\frac{(y_c {t_1})^2}{x_c {s_1}} \frac{\varphi_{t_1} (0)}{1 - \frac{\varphi_{t_1} (0)}{\varphi_{t_1}^{\{r\}}(z)}} \left(\varphi_{t_1}^{\{r+1\}} (z)- \varphi^{\{2\}}_{t_1}(0) \right)
\Bigg).
\end{align*}
Finally, summing over $q$ gives the joint generating function:
\begin{align*}
\mathbb E \Big[ 
s_1^{|\overline{H}_{r+1} (\mathbf \mathbf{T}_{\infty})|} s_2^{|\partial \overline{H}_{r+1} (\mathbf \mathbf{T}_{\infty})|} \Big]
& = \frac{t_1 s_2}{2} \varphi_{t_1} (0)
\frac{\mathrm d}{\mathrm{d}z} \Bigg(
\frac{\varphi_{t_1}^{\{r+1\}} (z)- \varphi^{\{2\}}_{t_1}(0) }
{1 - \frac{\varphi_{t_1} (0)}{\varphi_{t_1}^{\{r\}}(z)}}
\Bigg)_{z=\frac{s_1 s_2}{t_1}}.
\end{align*}
Note that this expression makes perfect sense for all choices of $s_1,s_2\in [0,1]$ since the radius of convergence of $\varphi_{t_1}^{\{r\}}$ is $s_1/t_1$.

To get a scaling limit, we set
\[
s_1 = \exp \left( - \frac{\lambda_1}{r^4} \right) \quad \text{and} \quad s_2 = \exp \left( - \frac{\lambda_1}{r^2} \right)
\]
and let $r$ go to infinity. For these particular choices, it is not hard to see that
\[
\varphi_{t_1}^{\{r\}} \left( \frac{s_1 s_2}{t_1} \right) \underset{r \to \infty}{\longrightarrow} 1
\]
and therefore, we have
\begin{align*}
\mathbb E \Big[ 
s_1^{|\overline{H}_{r+1} (\mathbf \mathbf{T}_{\infty})|} s_2^{|\partial \overline{H}_{r+1} (\mathbf \mathbf{T}_{\infty})|} \Big]
& \underset{r \to \infty}{\sim}
\frac{1}{2} \varphi (0)
\left(\frac{1}{1-\varphi(0)} - \varphi(0) \frac{1-\varphi^{\{2\}}(0)}{(1-\varphi(0))^2} \right)
\frac{\mathrm d\varphi_{t_1}^{\{r\}}}{\mathrm{d}z} \Bigg(
\frac{s_1 s_2}{t_1}
\Bigg)\\
& = 
\frac{\mathrm d\varphi_{t_1}^{\{r\}}}{\mathrm{d}z} \Bigg(
\frac{s_1 s_2}{t_1} \Bigg)
\end{align*}
with $\varphi(0) = 3/4$ and $\varphi^{\{2\}}(0) = 8/9$. With a bit more work (or our Maple worksheet!) we get the limit of this derivative and obtain the scaling limit:
\begin{align*}
\lim_{r \to \infty} 
\mathbb E &\Big[ 
\exp \left(- \lambda_1 r^{-4}|\overline{H}_{r} (\mathbf \mathbf{T}_{\infty})| - \lambda_2 r^{-2} |\partial \overline{H}_{r} (\mathbf \mathbf{T}_{\infty})| \right)\Big]\\
& \quad =
\frac
{
\left(\frac{2}{3} + \frac{\lambda_2}{(6 \lambda_1)^{1/2}} \right)^{-1/2}
\sinh \left( (6 \lambda_1)^{1/4}\right) + \cosh \left( (6 \lambda_1)^{1/4}\right)
}
{\Bigg( \left(\frac{2}{3} + \frac{\lambda_2}{(6 \lambda_1)^{1/2}} \right)^{1/2}
\sinh \left( (6 \lambda_1)^{1/4}\right) + \cosh \left( (6 \lambda_1)^{1/4}\right)\Bigg)^3}.
\end{align*}
This expression is still valid for $\lambda_1 \to 0$ giving
\begin{align*}
\lim_{r \to \infty} 
\mathbb E &\Big[ 
\exp \left(- \lambda_2 r^{-2} |\partial \overline{H}_{r} (\mathbf \mathbf{T}_{\infty})| \right)\Big]
=
\frac{1}{\left( 1 + \sqrt{\lambda_2}\right)^3}.
\end{align*}

\begin{remark} \label{rem:fullscaling} In fact, based on the results \cites{CLG16,Kr,M} it is likely that we can get a full scaling limit for the rescaled process of the volume and perimeter of the horohulls. More precisely we should have
$$ \left( \frac{| \overline{H}_{[rt]}|}{ r^{4}}, \frac{|\partial \overline{H}_{[rt]}|}{ r^{2}} \right) \xrightarrow[r\to\infty]{(d)} ( \mathcal{X}_{t}, \mathcal{V}_{t})_{t \geq 0}$$ where $ \mathcal{X}$ is a branching process with branching mechanism $\psi(u) = 2 u^{3/2}$  starting from $0$ conditioned to survive and conditionally on $ \mathcal{X}$ we have,
$$ \mathcal{V}_{t} = \sum_{s_{i} \leq t}  \frac{4}{3}\left(\Delta \mathcal{X}_{s_{i}}\right)^2 \xi_{i},$$ where $s_{i}$ is a enumeration of the the jump times of $ \mathcal{X}$ and $\xi_{i}$ are i.i.d.~random variables of law $ \frac{ \mathrm{d}x}{ \sqrt{2\pi x^{5}}}e^{-1/(2x)} \mathbf{1}_{x>0}$.
\end{remark}

\bigskip

\noindent \textsc{Nicolas Curien,\\ Laboratoire de Math\'ematiques d'Orsay, Univ. Paris-Sud, CNRS, Universit\'e Paris-Saclay, 91405 Orsay, France}

\bigskip

\noindent \textsc{Laurent M\'enard,\\ Laboratoire Modal'X, UPL, Univ. Paris Nanterre, F92000 Nanterre, France}

\end{document}